\documentclass[12pt]{article}
\usepackage{amsmath}
\usepackage{amssymb}
\usepackage{mathrsfs}
\numberwithin{equation}{section}
\usepackage{datetime}
\usepackage{enumerate}
\usepackage{diagbox}
\usepackage[linkcolor=black,urlcolor=black,colorlinks=true]{hyperref}
\usepackage{xcolor,colortbl}
\usepackage{tabularx}
\usepackage{yhmath}
\usepackage{graphicx,graphics}
\usepackage{epstopdf}
\usepackage{dsfont}
\def \beq {\begin{equation}}
\def \eeq {\end{equation}}
\def \ba {\begin{array}}
\def \ea {\end{array}}
\def \dis {\displaystyle}
\renewcommand{\Im}{\mbox{\rm Im}}
\renewcommand{\Re}{\mbox{\rm Re}}

\renewcommand{\r}{\mathop{\rightarrow}}
\def\r{\rightarrow}
\renewcommand{\r}{\mathop{\rightarrow}}
\newcommand{\fdem}{\hspace*{\fill}~$\Box$\par\endtrivlist\unskip}
     
\renewcommand{\P}{\mathbb{P}}

\def \div {\mbox{\rm div}\,}
\def \cqfd {\hfill$\Box$}
\def \al {\alpha}
\def \be {\beta}

\def \ep {\varepsilon}

\def \ph {\varphi}

\def \si {\sigma}

\newcommand{\N}{\mathbb{N}}     
\newcommand{\Z}{\mathbb{Z}}
\newcommand{\R}{\mathbb{R}} 
     
\newcommand{\C}{\mathbb{C}} 
\newcommand{\Q}{\mathbb{Q}}

\newcommand{\dT}{\mathbb{T}}
\newcommand{\cha}{\mathds{1}}

\def \cD {\mathscr{D}}
\def \cE {\mathscr{E}}

\def \cI {\mathscr{I}}

\def \cM {\mathscr{M}}

\def \sfC {\mathsf{C}}

\def \sfE {\mathsf{E}}

\def \beq {\begin{equation}}
\def \eeq {\end{equation}}
\def \ba {\begin{array}}
\def \ea {\end{array}}
\def \bs {\bigskip}

\def \ecart {\noalign{\medskip}}

\newenvironment{proof}[1]{\textit{Proof#1.\,}}{\fdem}

\newtheorem{atheo}{Theorem}[section]

\newtheorem{alem}{Lemma}[section]
\newtheorem{arem}{Remark}[section]
\newtheorem{Aexa}{Exemple}[section]

\newtheorem{apro}[alem]{Proposition}

\title{Specific properties of the ODE's flow in
\\
dimension two {\em versus} dimension three}
\author{\large Marc Briane \& Lo\"\i c Herv\'e
\\*[.1em]
\normalsize Univ Rennes, INSA Rennes,  CNRS, IRMAR - UMR 6625, F-35000 Rennes, France
\\*[.1em]
\normalsize mbriane@insa-rennes.fr \& loic.herve@insa-rennes.fr
}
\begin{document}
\maketitle
\tableofcontents
\begin{abstract}
This paper deals with the asymptotics of the ODE's flow induced by a regular vector field $b$ on the $d$-dimensional torus $\R^d/\Z^d$.
First, we start by revisiting the Franks-Misiurewicz theorem which claims that the Herman rotation set of any two-dimensional continuous flow is a closed line segment of $\R^2$. Various general examples illustrate this result, among which a complete study of the Stepanoff flow associated with a vector field $b=a\,\zeta$, where $\zeta$ is a constant vector in $\R^2$.
Furthermore, several extensions of the Franks-Misiurewicz theorem are obtained in the two-dimensional  ODE's context. On the one hand, we provide some interesting stability properties in the case where the Herman rotation set has a commensurable direction. On the other hand, we present new results highlighting the exceptional character of the opposite case, {\em i.e.} when the Herman rotation set is a closed line segment with $0_{\R^2}$ at one end and with an irrational slope, if it is not reduced to a single point.
Besides this, given a pair $(\mu,\nu)$ of invariant probability measures for the flow, we establish new Fourier relations between the determinant $\det\,(\widehat{\mu b}(j),\widehat{\nu b}(k))$ and the determinant $\det\,(j,k)$ for any pair $(j,k)$ of non null integer vectors, which can be regarded as an extension of the Franks-Misiurewicz theorem.
Next, in contrast with dimension two, any three-dimensional closed convex polyhedron with rational vertices is shown to be the rotation set associated with a suitable vector field~$b$.
Finally, in the case of an invariant measure $\mu$ with a regular density and a non null mass $\mu(b)$ with respect to~$b$, we show that the homogenization of the two-dimensional transport equation with the oscillating velocity $b(x/\ep)$ as $\ep$ tends to $0$, leads us to a nonlocal limit transport equation, but with the effective constant velocity~$\mu(b)$.
\end{abstract}
\par\bs\noindent
\par\bs\noindent
%%%%%%%%%%
\section{Introduction}
Let $b$ be a  $C^1$-regular $d$-dimensional vector field defined on the torus $Y_d:=\R^d\setminus\Z^d$.
%$Y_d:=\R^d\setminus\Z^d$.
In this paper, we study the asymptotics of the associated ODE's flow $X(\cdot,x)$ for $x\in Y_d$, defined by
\beq\label{bX}
\left\{\ba{ll}
\dis {\partial X\over\partial t}(t,x)=b(X(t,x)), & t\in\R
\\ \ecart
X(0,x)=x.
\ea\right.
\eeq
Our aim is to characterize the best as possible the asymptotics of the flow $X$, and among other the set $\rho(b)$, according to Misiurewicz and Ziemian \cite[(1.1)]{MiZi1}, composed of all the limit points related to the sequences
\beq\label{asyX}
{X(t_n,x_n)\over t_n}\quad\mbox{for any }t_n>0,\ \mbox{and any }x_n\in Y_d.
\eeq
By virtue of  \cite[Theorem~2.4, Remark~2.5, Corollary~2.6]{MiZi1} it turns out that the convex hull of~$\rho(b)$ agrees with the Herman rotation set $\sfC_b$ \cite{Her} composed of the vector masses
\beq\label{mub}
\mu(b)=\int_{Y_d}b(x)\,d\mu(x),
\eeq
with respect to the probability measures $\mu$ on $Y_d$ which are invariant for the flow, {\em i.e.} for any continuous function $\varphi$ in $Y_d$,
\beq\label{invmuph}
\forall\,t\in\R,\quad\int_{Y_d}\varphi(X(t,x))\,d\mu(x)=\int_{Y_d}\varphi(x)\,d\mu(x).
\eeq
The two-dimensional case is quite specific due to Poincar\'e-Bendixon's theory (see, {\em e.g.}, \cite[Chapter~VII]{Har}) combined with Siegel's curve theorem~\cite{Sie}. So in dimension two, assuming that the vector field $b$ is non vanishing in $Y_2$, Peirone~\cite{Pei} proved that the limit \eqref{asyX} actually exists for any point $x\in Y_2$, but he gave a counterexample in dimension three. In \cite{Pei} it can be noticed that Herman's rotation set $\sfC_b$ is either a unit set, when the flow has no periodic orbit in~$Y_2$, or a closed line segment (see Proposition~\ref{pro.Cb2d}). This collinearity result can  be also observed in other examples \cite{Tas,BrHe1,BrHe2}, and it is illustrated by Example~\ref{exa.pre}.
Actually, the result is more general, since Franks and Misiurewicz \cite[Theorem~1.2]{FrMi} proved that the rotation set of any two-dimensional continuous flow is always a closed line segment of a line passing through the null vector $0_{\R^2}$.
However, the situation is quite different for a general lift $F:\R^2\to\R^2$ (through the canonical surjection $\Pi:\R^2\to Y_2$) of some homeomorphism on $Y_2$ homotopic to identity, satisfying
\beq\label{lift}
\forall\,k\in\Z^2,\ \forall\,x\in\R^2,\quad F(x+k)=F(x)+k.
\eeq
Indeed, Kwapisz~\cite{Kwa} proved that any convex polygone of vertices at rational points of $\R^2$, is a rotation set of some suitable lift $F:\R^2\to\R^2$.
In the present context of ODE \eqref{bX} the time-$1$ flow $F=X(1,\cdot)$ is such a lift.
To deepen our exploration, we show (see Proposition~\ref{pro.FX} and Example~\ref{exa.LM}) that the two-dimensional lift $F$ introduced by Llibre and Mackay \cite[Example~2]{LlMa} (which is also revisited in \cite{MiZi2}) has the whole square $[0,1]^2$ as a rotation set, but it cannot be represented by any flow $X(1,\cdot)$ solution to~\eqref{bX}.
Actually, the set of the time-$1$ flows $X(1,\cdot)$ solutions to \eqref{bX} is strictly contained in the set of the lifts $F$ of homeomorphisms on $Y_2$ homotopic to identity and satisfying \eqref{lift}.
\par
In Section~\ref{ss.revFM} we first give a partial proof of the Franks-Miziurewicz collinearity result (see Proposition~\ref{pro.FMcol}) assuming that one of the two invariant probability measures is regular.
Then, revisiting different works \cite{Oxt1,Aku,Str1,Str2,Fra1,Fra2,Har,Pei,BrHe1} we recover for the specific two-dimensional ODE's flow (see Proposition~\ref{pro.Cb2d}) the alternatives satisfied by the rotation set $\sfC_b$ -- obtained in \cite[Theorem~1.2]{FrMi} (recalled in Theorem~\ref{thm.FM} below) for any continuous flow -- in the following more accurate picture:
\begin{itemize}
\item[$({\rm I})$] \vskip -.2cm
$\sfC_b = I \,\zeta$, where $I$ is a closed segment of $\R$ not reduced to a single point, and $\zeta$ is a commensurable vector of $\R^2$.
% This holds if the flow $X$ \eqref{bX} has two infinite periodic compact orbits in $Y_2$ with distinct asymptotics.

\item[$({\rm II})$] \vskip -.2cm
$\sfC_b = \{\zeta\}$, where $\zeta$ is a non null commensurable vector of $\R^2$.
% In this case, the vector field~$b$ does not vanish in $Y_2$, and the flow $X$ has at least one infinite periodic compact orbit in~$Y_2$, and all the infinite periodic orbits have $\zeta$ as asymptotics.

\item[$({\rm III})$] \vskip -.2cm
$\sfC_b = \{\zeta\}$, where $\zeta$ is an incommensurable vector of $\R^2$.
% if, and only if, $b$ does not vanish in~$Y_2$, and the flow $X$ has no periodic compact orbit in $Y_2$. For the {\em if}, we assume in addition that $b\in C^3_\sharp(Y_2)^3$.
 
 \item[$({\rm IV})$] \vskip -.2cm
 $\sfC_b$ is either $\{0_{\R^2}\}$, or $\sfC_b = I\,\zeta$, where $I$ is a closed segment of $\R$ such that $0\in I\neq\{0\}$, and $\zeta$ is an incommensurable vector of $\R^2$.
\end{itemize}
\par
Next, in Section~\ref{ss.morerotset} we study several extensions of the Franks-Misiurewicz theorem (see Proposition~\ref{pro.Cb2dcom}, Proposition~\ref{pro.FMincom} and Theorem~\ref{thm.Ssgn}), which are specific to the present two-dimensional ODE's context, and which are new to the best of our knowledge. First of all, some stability properties in connection  with case~(I) are investigated for non vanishing vector fields~$b$ (see Proposition~\ref{pro.Cb2dcom}). More precisely, writing $b=\rho\, \Phi$ with $\rho$ a positive function in $C^1(Y_2)$ and $\Phi$ a non vanishing vector field in $C^1(Y_2)^2$, and assuming that the rotation set $\sfC_b$ satisfies case~(I), we prove that the direction of $\sfC_b$ only depends on $\Phi$, and that the perturbed rotation set $\sfC_{\tilde{b}}$ with $\widetilde{b}=\widetilde{\rho}\,\Phi$, still satisfies case~(I) provided that the uniform norm of $|\rho-\widetilde{\rho}|$ is small enough.
On the other hand, Franks and Misiurewicz proved that in the incommensurable case~(IV) the null vector $0_{\R^2}$ is always an end of the closed segment $\sfC_b$, and they mentioned that it is delicate to find examples of such a situation.
Then, rather than proving this incommensurable case in the ODE's context (see Remark~\ref{rem.Qincom} on this point), we provide (see Proposition~\ref{pro.FMincom} and Theorem~\ref{thm.Ssgn}) general classes of vectors fields $b$ which are in some sense complementary of the incommensurable case:
\begin{itemize}
\item First, we consider (see Proposition~\ref{pro.FMincom} $(ii)$) the vector fields of type $b=a\,\Phi$, where $a$ is a vanishing non negative function in $C^1(Y_2)$ with $1/a\in L^1(Y_2)$, and  $\Phi$ is a non vanishing divergence free vector field in $C^1(Y_2)^2$.
Assuming that the mean value $\overline{\Phi}$ of $\Phi$ is incommensurable in~$\R^2$, we prove that the rotation set $\sfC_b$ is the closed line segment $[0_{\R^2},\overline{\Phi}]\neq\{0_{\R^2}\}$.
In this context, we may observe (see Remark~\ref{rem.Sinc}) that the vector field $b$ cannot be in $C^2(Y_2)^2$, which shows the exceptional character of the closed line segment occurrence in case~(IV).
\item Second, we consider (see Theorem~\ref{thm.Ssgn}) the vector fields of type $b=a\,\Phi$, where $a$ is a changing sign function in $C^1(Y_2)$, and $\Phi$ is a non vanishing vector field in $C^3(Y_2)^2$ having a positive invariant probability measure $\sigma(x)\,dx$ with $\sigma\in C^3(Y_2)$.
Assuming that the mean value of $\sigma\,\Phi$ is incommensurable in $\R^2$, we prove that the support of any invariant probability measure for the flow $X$ is contained in the set $\{a=0\}$, so that the rotation set $\sfC_b$ is reduced to $\{0_{\R^2}\}$.
\end{itemize}
These two results illuminate the alternative of the incommensurable case~(IV).
More precisely, according to the incommensurable case of the Franks-Misiurewicz \cite[Theorem~1.2]{FrMi}, even the change of sign of the function $a$ in $b=a\,\Phi$ cannot allow $0_{\R^2}$ to be an interior point of the closed line segment $\sfC_b$, contrary to the commensurable case~(I).
\par
In Section~\ref{ss.Sflow} we fully illustrate the cases (I), (II), (III), (IV) and the previous results by presenting a complete picture of the Stepanoff flow~\cite{Ste,Oxt2} associated with vector fields of type $b=a\,\zeta$, where $a$ is a function in $C^1(Y_2)$ and $\zeta$ is a unit vector of $\R^2$.
In particular, if the vector $\zeta$ is incommensurable in $\R^2$,  the case where $\sfC_{a\,\zeta}$ is a closed line segment not reduced to $\{0_{\R^2}\}$ only holds when $a$ has a constant sign and $1/a\in L^1(Y_2)$.
This again highlights the exceptionality of this case, since $a$ cannot be actually in $C^2(Y_2)$ (see Remark~\ref{rem.Sinc}).
\par
In Section~\ref{s.intFrel}, we prove (see Theorem~\ref{thm.relmunu}) the following original, up to our best knowledge, integral relation satisfied by any pair $(\mu,\nu)$ of invariant probability measures for the flow $X$ and any function $\rho\in C^2(Y_2\!\times\!Y_2)$,
 \beq\label{relation}
\ba{l}
\dis \int_{Y_2}\kern -.2em\int_{Y_2}\rho(x,y)\det\,(b(x),b(y))\,d\mu(x)\,d\nu(y)
\\ \ecart
\dis =R_\perp\nu(b)\cdot\int_{Y_2}\kern -.2em\left(\int_{Y_2}\rho(x,y)\,dy\right)b(x)\,d\mu(x)
-R_\perp\mu(b)\cdot\int_{Y_2}\kern -.2em\left(\int_{Y_2}\rho(x,y)\,dx\right)b(y)\,d\nu(y)
\\ \ecart
\dis +\int_{Y_2}\kern -.2em\int_{Y_2}\kern -.2em\left({\partial^2\rho\over\partial x_1\partial y_2}-{\partial^2\rho\over\partial x_2\partial y_1}\right)u^\sharp(x)\,v^\sharp(y)\,dxdy,
\ea
\eeq
where $u^\sharp$ and $v^\sharp$ are the stream functions with bounded variation in $Y_2$, satisfying the vector-valued measure representations
\beq\label{bmuusi}
b\,\mu=\mu(b)+\nabla^\perp u^\sharp\quad\mbox{and}\quad b\,\nu=\nu(b)+\nabla^\perp v^\sharp\quad\mbox{in }Y_2,
\eeq
where $\nabla^\perp$ denotes the orthogonal gradient.
The key-ingredient of relation \eqref{relation} is the Liouville theorem for invariant measures (see, {\em e.g.}, \cite[Theorem~1, Section~2.2]{CFS}) which is considered as a divergence-curl result for invariant measures (see Proposition~\ref{pro.divcurl} and \cite[Proposition~2.2]{BrHe1}), and which is combined with the representation of a divergence free field by an orthogonal gradient in dimension two. Such a representation does not hold in dimension three, so that there is no three-dimensional relation similar to~\eqref{relation}. Moreover, it turns out that the integral relation~\eqref{relation} is equivalent to the Fourier relations
\beq\label{Fmubnubi}
\forall\,(j,k)\in(\Z^2\setminus\{0_{\R^2}\})^2\cup\{(0_{\R^2},0_{\R^2})\},\quad \det\big(\widehat{\mu b}(j),\widehat{\nu b}(k)\big)
=-\,4\pi^2\det\,(j,k)\,\widehat{u^\sharp}(j)\,\widehat{v^\sharp}(k).
\eeq 
Relations~\eqref{Fmubnubi} show that the vector-valued measures $\mu\,b,\nu\,b$ for any pair $(\mu,\nu)$ of invariant probability measures for the flow \eqref{bX}, are strongly correlated through their Fourier coefficients.
Noting that the case $j=k=0_{\R^2}$ in \eqref{Fmubnubi} agrees with the Franks-Misiurewicz collinearity result \cite[Theorem~1.2]{FrMi}, formulas \eqref{Fmubnubi} with non null collinear integer vectors may be thus regarded (see Remark~\ref{rem.Fcoef}) as an extension of the collinearity result to the specific two-dimensional ODE's flow \eqref{bX}.
Relations~\eqref{Fmubnubi} do not hold true in general if only one of the two integer vectors is null.
However, if the direction of the rotation set $\sfC_b$ is incommensurable in $\R^2$ with $1/|b|$ in $L^1(Y_2)$, then the extension of equalities \eqref{Fmubnubi} to any pair of integer vectors characterizes surprisingly some Stepanoff flows (see Proposition~\ref{pro.0hmub})
\par
In Section~\ref{s.3d}, contrary to the former two-dimensional results, we prove (see Theorem~\ref{thm.Cb3d}) that each closed convex polyhedron of $\R^3$ is a rotation set $\sfC_b$ for a suitable vector field $b$.
\par
Finally, in Section~\ref{s.hom} the collinearity result satisfied by Herman's rotation set is applied to the homogenization of the two-dimensional transport equation with an oscillating velocity $b(x/\ep)$ as $\ep\to 0$.
This homogenization problem has been the subject of several works \cite{Bre,Gol1,Gol2,HoXi,Tas} under the ergodic assumption, and more recently \cite{Bri2,BrHe1,BrHe2} with a non ergodic assumption.
If these assumptions hold true, the homogenized equation turns to be a transport equation with a constant effective velocity.
Otherwise, it is known \cite{Tar,AHZ1,AHZ2} that the homogenized equation involves a nonlocal term.
Here, assuming that the flow $X$ has an invariant measure $\mu$ with $\mu(b)\neq 0_{\R^2}$ and a regular enough density, we derive (see Theorem~\ref{thm.homtreq}) a nonlocal homogenized equation, but with the effective constant velocity $\mu(b)$.
\subsection{Notation}\label{ss.not}
\begin{itemize}
\item $(e_1,\dots,e_d)$ denotes the canonical basis of $\R^d$, and $0_{\R^d}$ denotes the null vector of $\R^d$.
\item $I_d$ denotes the unit matrix of $\R^{d\times d}$, and $R_\perp$ denotes the $(2\times 2)$ rotation matrix $\,\footnotesize{\begin{pmatrix}0 & 1 \\ -1 & 0\end{pmatrix}}$.
\item $``\cdot"$ denotes the scalar product and $|\cdot|$ the euclidean norm in $\R^d$.
\item $|A|$ denotes the Lebesgue measure of any measurable set in $\R^d$ or $Y_d$.
\item $Y_d$, $d\geq 1$, denotes the $d$-dimensional torus $\R^d/\Z^d$ (which may be identified to the unit cube $[0,1)^d$ in $\R^d$), and $0_{Y_d}$ denotes the null vector of $Y_d$.
\item $\Pi$ denotes the canonical surjection from $\R^d$ on $Y_d$.
\item $d_{Y_2}$ denotes the distance function to the point $0_{Y_2}$ in $Y_2$, {\em i.e.}
\beq\label{dY2}
d_{Y_2}(x):=\min_{k\in\Z^2}|x-k|\leq{1\over\sqrt{2}}\quad\mbox{for }x\in Y_2.
\eeq
\item $\cha_A$ denotes the characteristic function of a set $A$, and ${\rm Id}:E\to E$ denotes the identity function in a set $E$ given by the context.
\item $C^k_c(\R^d)$, $k\in\N\cup\{\infty\}$, denotes the space of the real-valued functions in $C^k(\R^d)$ with compact support in $\R^d$.
\item $C^k_\sharp(Y_d)$, $k\in\N\cup\{\infty\}$, denotes the space of the real-valued functions $f\in C^k(\R^d)$ which are $\Z^d$-periodic, {\em i.e.}
\beq\label{fper}
\forall\,k\in\Z^d,\ \forall\,x\in \R^d,\quad f(x+k)=f(x).
\eeq
\item The jacobian matrix of a $C^1$-mapping $F:\R^d\to\R^d$ is denoted by the matrix-valued function $\nabla F$ with entries $\dis {\partial F_i\over\partial x_j}$ for $i,j\in\{1,\dots,d\}$.
\item The abbreviation ``a.e.'' for almost everywhere, will be used throughout the paper.
The simple mention ``a.e.'' refers to the Lebesgue measure on $\R^d$.
\item $dx$ or $dy$ denotes the Lebesgue measure on $\R^d$.
\item For a Borel measure $\mu$ on $Y_d$, extended by $\Z^d$-periodicity to a Borel measure $\widetilde{\mu}$ on $\R^d$, a $\widetilde{\mu}$-measurable function $f:\R^d\to\R$ is said to be $\Z^d$-periodic $\widetilde{\mu}$-a.e. in $\R^d$, if
\beq\label{fpermu}
\forall\,k\in\Z^d,\quad f(\cdot+k)=f(\cdot)\;\;\mbox{$\widetilde{\mu}$-a.e. on }\R^d.
\eeq
\item For a Borel measure $\mu$ on $Y_d$, $L^p_\sharp(Y_d,\mu)$, $p\geq 1$, denotes the space of the $\mu$-measurable functions $f:Y_d\to\C$ such that
\[
\int_{Y_d}|f(x)|^p\,d\mu(x)<\infty.
\]
\item $L^p_\sharp(Y_d)$, $p\geq 1$, simply denotes the space of the Lebesgue measurable functions $f$ in $L^p_{\rm loc}(\R^d)$, which are $\Z^d$-periodic $dx$-a.e. in $\R^d$.
\item  $\cM_{\rm loc}(\R^d)$ denotes the space of the non negative Borel measures on $\R^d$, which are finite on any compact set of $\R^d$.
\item $\cM_\sharp(Y_d)$ denotes the space of the non negative Radon measures on $Y_d$, and $\cM_p(Y_d)$ denotes the space of the probability measures on $Y_d$.
\item $\cD'(\R^d)$ denotes the space of the distributions on $\R^d$.
\item $BV_\sharp(Y_d)$ denotes the space of the functions $f\in BV_{\rm loc}(\R^d)$ ({\em i.e.} with bounded variation locally in $\R^d$) which are $\Z^d$-periodic a.e. in $\R^d$,
namely $f\in L^2_\sharp(Y_d)$ and $\nabla f\in\cM_\sharp(Y_d)^d$.
\item For a Borel measure $\mu$ on $Y_d$ and for $f\in L^1_\sharp(Y_d,\mu)$, we denote
\beq\label{muf}
\mu(f):=\int_{Y_d}f(x)\,d\mu(x),
\eeq
which is simply denoted by $\overline{f}$ when $\mu$ is Lebesgue's measure.
The same notation is used for a vector-valued function in $L^1_\sharp(Y_d,\mu)^d$. 
If $f$ is non negative, its harmonic mean $\underline{f}$ is defined by
\[
\underline{f}:=\left(\int_{Y_d}{dy\over f(y)}\right)^{-1}.
\]
\item For a given measure $\lambda\in\cM_\sharp(Y_d)$, the Fourier coefficients of $\lambda$ are defined by
\[
\widehat{\lambda}(k):=\int_{Y_d}e^{-2i\pi k\cdot x}\,d\lambda(x)\quad\mbox{for }k\in\Z^d.
\]
The same notation is used for a vector-valued measure in $\cM_\sharp(Y_d)^d$.
\item $c$ denotes a positive constant which may vary from line to line.
%\item The notation $\cI_b$ in \eqref{Ib} will be used throughout the paper.
\end{itemize}
\subsection{A few tools of ergodic theory}\label{ss.erg}
Let $b\in C^1_\sharp(Y_d)^d$.
A measure $\mu\in\cM_p(Y_d)$ on $Y_d$ is said to be {\em invariant} for the flow $X$ defined by \eqref{bX} if
\beq\label{invmu}
\forall\,t\in\R,\ \forall\,\psi\in C^0_\sharp(Y_d),\quad\int_{Y_d}\psi\big(X(t,y)\big)\,d\mu(y)=\int_{Y_d}\psi(y)\,d\mu(y).
\eeq
An invariant probability measure $\nu$ for the flow $X$ is said to be {\em ergodic}, if
\beq\label{Xerg}
\forall\,f\in L^1_\sharp(Y_d,\nu),\mbox{ invariant for $X$ w.r.t. $\nu$},\quad f=\nu(f)\mbox{ $\nu$-a.e. in $Y_d$}.
\eeq
Then, define the set of invariant measures
\beq\label{Ib}
\cI_b:=\big\{\mu\in\cM_p(Y_d): \mu\mbox{ invariant for the flow }X\big\},
\eeq
and the subset of $\cI_b$ of composed of the ergodic measures
\beq\label{cEb}
\cE_b:=\big\{\nu\in\cI_b: \nu\mbox{ ergodic for the flow }X\big\}.
\eeq
It is known that the ergodic measures for the flow $X$ are the extremal points of the convex set $\cI_b$ so that
\beq\label{IbEb}
\cI_b={\rm conv}(\cE_b).
\eeq
Also define for any vector field $b\in C^1_\sharp(Y_d)^d$ the following non empty subsets of $\R^d$:
\begin{itemize}
\item According to \cite[(1.1)]{MiZi1} the set of all the limit points of the sequences $\big(X(n,x_n)/n\big)_{n\geq 1}$ in $\R^d$ obtained for any sequence $(x_n)_{n\geq 1}$ in $Y_d$, is defined by
\beq\label{rhob}
\rho(b):=\bigcap_{n\geq 1}\left(\overline{\bigcup_{x\in Y_d}\left\{{X(k,x)\over k}:k\geq n\right\}}\right).
\eeq
By \cite[Lemma~2.2, Theorem~2.3]{MiZi1} it is a compact and connected set of $\R^d$.
\item The so-called Herman \cite{Her} rotation set is defined by
\beq\label{Cb}
\sfC_b:=\big\{\mu(b):\mu\in\cI_b\big\}.
\eeq
It is clearly a compact and convex set of $\R^d$.
\item By restriction to ergodic measures define the subset of $\sfC_b$
\beq\label{Eb}
\sfE_b := \big\{\nu(b),\ \nu\in \cE_b\big\}.
\eeq
\end{itemize}
An implicit consequence of \cite[Theorem~2.4, Remark~2.5, Corollary~2.6]{MiZi1} shows that
\beq\label{rhobCb}
\rho(b)\subset \sfC_b={\rm conv}\,(\rho(b)) \quad\mbox{and}\quad \# \rho(b)=1\Leftrightarrow \# \sfC_b=1.
\eeq
\begin{arem}\label{rem.X1}
Note that for any $k\in\N$, $X(k,\cdot)$ agrees with the $k$-th iteration of $X(1,\cdot)$ in the definition \eqref{rhob} of $\rho(b)$.
Hence, the equality $\sfC_b={\rm conv}\,(\rho(b))$ in \eqref{rhobCb} shows that the rotation set $\sfC_b$ is completely characterized by the flow $X(1,\cdot)$ and its iterations.
This characterization is still more flagrant in dimension two, since by virtue of \cite[Theorem~3.4~$(b)$]{MiZi1} we have $\sfC_b=\rho(b)$.
\end{arem}
In view of \eqref{IbEb} and \eqref{Eb} we also have (see Remark~\ref{rem.extCb} below)
\beq\label{CbEb}
\sfC_b = \mbox{\rm conv}\, (\sfE_b).
\eeq
Note that by definition~\eqref{rhob} the equivalence of \eqref{rhobCb} can be written for any $\zeta\in \R^d$,
\beq\label{Cbsindis}
\sfC_b=\{\zeta\}\quad\Leftrightarrow\quad\forall\,x\in Y_d,\;\; \lim_{n\to\infty}{X(n,x)\over n}=\zeta\quad\Leftrightarrow\quad
\forall\,x\in Y_d,\;\; \lim_{t\to\infty}{X(t,x)\over t}=\zeta.
\eeq
The second equivalence of \eqref{Cbsindis} is an easy consequence of the semi-group property satisfied by the flow, {\em i.e.}
\beq\label{sgX}
\forall\,x\in Y_d,\ \forall\,s,t\in\R,\quad X(s, X(t,x))=X(s+t,x).
\eeq
\par
We have the following characterization of an invariant measure known as Liouville's theorem, which can also be regarded as a divergence-curl result with measures (see Proposition~\ref{pro.divcurl} and \cite[Remark~2.2]{BrHe1} for further details).
\begin{apro}[Liouville's theorem]\label{pro.divcurl}
Let $b\in C^1_\sharp(Y_d)^d$, and let $\mu\in\cM_\sharp(Y_d)$. We define the Borel measure $\widetilde{\mu}\in\cM_{\rm loc}(\R^d)$ on~$\R^d$ by
\beq\label{tmu}
\int_{\R^d}\ph(x)\,d\widetilde{\mu}(x)=\int_{Y_d} \ph_\sharp(y)\,d\mu(y),\quad\mbox{where}\quad \ph_\sharp(\cdot):=\sum_{k\in\Z^d}\ph(\cdot+k)
\quad\mbox{for }\ph\in C^0_c(\R^d).
\eeq
Then, the three following assertions are equivalent:
\begin{enumerate}[$(i)$]
\item $\mu$ is invariant for the flow $X$, {\em i.e.} \eqref{invmu} holds,
\item $\widetilde{\mu}\,b$ is divergence free in the space $\R^d$, {\em i.e.}
\beq\label{dtmub=0}
\div(\widetilde{\mu}\,b)=0\quad\mbox{in }\cD'(\R^d),
\eeq
\item $\mu\,b$ is divergence free in the torus $Y_d$, {\em i.e.}
\beq\label{dmub=0}
\forall\,\psi\in C^1_\sharp(Y_d),\quad \int_{Y_d} b(y)\cdot\nabla\psi(y)\,d\mu(y)=0.
\eeq
\end{enumerate}
\end{apro}
\begin{arem}\label{rem.extCb}
Let $b\in C^1_\sharp(Y_d)^d$.
It turns out that for any extremal point $\xi$ of $\sfC_b$, there exists an ergodic measure $\nu\in \cE_b$ \eqref{cEb} such that $\nu(b)=\xi$. Hence, $\sfC_b$ is the convex hull of the set $\sfE_b$ \eqref{Eb}. Equality \eqref{CbEb} is stated without proof in \cite[Remark~2.5]{MiZi1}.
However, although the set $\cI_b$ is the closure of the convex hull of $\cE_b$, equality \eqref{CbEb} does not seem obvious to us. For reader's convenience a complete proof of \eqref {CbEb} is given in Appendix~\ref{a.CbEb}.
\end{arem}
\subsection{The Franks-Misiurewicz result for 2D continuous flows}\label{ss.conj}
Franks and Misiurewicz \cite[Theorem~1.2]{FrMi} proved that the rotational set of any two-dimensional continuous flow is composed of collinear vectors, which yields the following result in the specific case of the ODE's flow.
\begin{atheo}[\cite{FrMi}, Theorem~1.2]\label{thm.FM}
Let $b$ be a two-dimensional vector field in $C^1_\sharp(Y_2)^2$.
Then, we have
\beq\label{FMcol}
\forall\,(\mu,\nu)\in\cI_b\times\cI_b,\quad \det\,(\mu(b),\nu(b))=0,
\eeq
{\em i.e.} Herman's rotation set $\sfC_b$ is a closed line segment included in a vector line of $\R^2$.
\par\noindent
More precisely, the rotation set satisfies one of the three following assertions:
\begin{itemize}
\item[$(a)$] \vskip -.2cm
$\sfC_b$ is a unit set of $\R^2$,

\item[$(b)$] \vskip -.2cm
$\sfC_b$ is a closed segment of a line passing through $0_{\R^2}$ and a non null point of $\Q^2$,

\item[$(c)$] \vskip -.2cm
$\sfC_b$ is a closed line segment $[0_{\R^2},\zeta]$ with $\zeta$ an incommensurable vector of $\R^2$.
\end{itemize}
\end{atheo}
\par
On the one hand, in Proposition~\ref{pro.FMcol} below we partly recover the collinearity result~\eqref{FMcol}, assuming the regularity of one of the two invariant probability measures through a quite different approach which is adapted to the more specific ODE's flow~\eqref{bX}.
On the other hand, in Proposition~\ref{pro.Cb2d} below we revisit the cases $(a)$, $(b)$, $(c)$ of Theorem~\ref{thm.FM}, and as a by-product we find the collinearity result \eqref{FMcol}, using essentially the Peirone result \cite[Theorem~3.1]{Pei} (widely later than \cite[Theorem~1.2]{FrMi}) combined with the Franks results \cite[Theorem~3.2]{Fra2} and \cite[Theorems~3.5]{Fra1} (earlier than \cite[Theorem~1.2]{FrMi}).
\begin{arem}\label{rem.Cbseg}
As a consequence of the collinearity property \eqref{FMcol} satisfied by Herman's rotation set~$\sfC_b$, there exists a non null vector $\xi\in\R^2$ such that the flow $X$ satisfies the unidirectional asymptotics
\beq\label{Xxi}
\forall\,x\in Y_2,\quad\lim_{t\to\infty}{X(t,x)\cdot\xi\over t}=0.
\eeq
Note that the scalar asymptotics \eqref{Xxi} holds true everywhere and not only almost-everywhere, unlike in Birkhoff's theorem.
\par\noindent
Indeed, if the rotation set $\sfC_b$ is the unit set $\{0_{\R^2}\}$, then from \eqref{Cbsindis} we deduce that
\[
\forall\,x\in Y_d,\quad \lim_{t\to\infty}{X(t,x)\over t}=0_{\R^2},
\]
which implies \eqref{Xxi} for any $\xi\in\R^2$.
Otherwise,  the rotation set $\sfC_b$ reads as $I\,\zeta$, where $I$ is a closed segment of $\R$ and $\zeta$ is a non null vector of $\R^2$.
Then, since by \eqref{rhob}, \eqref{rhobCb} $\sfC_b$ is the convex hull of the limit points of all the sequences $(X(t_n,x)/t_n)_{n\in\N}$ for any positive sequence $(t_n)_{n\in\N}$ converging to $\infty$ and any $x\in Y_2$, we get immediately the limits~\eqref{Xxi} with the non null vector $\xi:=R_\perp\zeta$ orthogonal to $\zeta$.
\end{arem}
%%%%%%%%%%
\section{Various properties of the two-dimensional ODE's flow}
To motivate this section let us start by the two following examples which try to show the specificity of the two-dimensional ODE's flow compared to general continuous flows.
\subsection{Some preliminary examples}
\subsubsection{Specificity of the ODE's flow among continuous flows}\label{ss.speODE}
The first Example~\ref{exa.LM} below is based on the two-dimensional \cite[Example~2]{LlMa} in which the lift~$F$ satisfying \eqref{lift} has the whole square $[0,1]^2$ as a rotation set $\rho(F)$ (see \cite[(1.1)]{MiZi1} and \cite{MiZi2}).
However, we will show that this lift $F$ does not agree with any time-$1$ flow $X(1,\cdot)$ solution to~\eqref{bX}, which is consistent with the two-dimensional result~\eqref{FMcol} on continuous flows.
To this end, we will use the following result which provides a necessary condition for a lift $F$ for agreeing with $X(1,\cdot)$ for some flow $X$ solution to~\eqref{bX}.
\begin{apro}\label{pro.FX}
Let $F\in C^1(\R^d)^d$ be a lift satisfying \eqref{lift}.
Then, if there exists a flow $X$ \eqref{bX} associated with some vector field $b\in C^1_\sharp(Y_d)^d$ such that $F=X(1,\cdot)$ in $Y_d$, then the following relation holds
\beq\label{FbX}
(\nabla F)\,b=b\circ F\quad\mbox{in }Y_d.
\eeq
Moreover, we have
\beq\label{Fa}
\forall\,a\in \R^d\mbox{ \rm s.t. }F(a)-a\in\Z^d,\quad\det\,(\nabla F(a)-I_d)\neq 0\;\Rightarrow\;b(a)=0_{\R^d}.
\eeq
\end{apro}
\begin{proof}{ of Proposition~\ref{pro.FX}}
Let $F\in C^1(\R^2)^2$ satisfying \eqref{lift}, and let $X$ be a flow associated with a vector field $b\in C^1_\sharp(Y_d)^d$ by \eqref{bX}.
Replacing any point $x\in Y_d$ by $X(t,x)$ for any $t\in\R$, and taking the derivative with respect to $t$, we get that
\[
\ba{rrl}
& \forall\,x\in Y_d, & F(x)=X(1,x)
\\ \ecart
\Leftrightarrow & \forall\,t\in\R,\ \forall\,x\in Y_d, & F(X(t,x))=X(1,X(t,x))=X(t+1,x)
\\ \ecart
\Rightarrow & \forall\,t\in\R,\ \forall\,x\in Y_d, & \big(\nabla F(X(t,x))\big)\,b(X(t,x))=b(X(t+1,x))
\\ \ecart
(\mbox{with }t=0)\quad\Rightarrow & \forall\,x\in Y_d, & (\nabla F(x))\,b(x)=b(X(1,x))=b(F(x)),
\ea
\]
or equivalently \eqref{FbX}, which proves the first assertion.
\par
Now, assume that \eqref{FbX} holds, and let $a\in Y_d$ satisfying the left hand-side of the implication \eqref{Fa}.
Then, since $b(F(a))=b(a)$, we have immediately
\[
(\nabla F(a)-I_d)\,b(a)=0_{\R^d},
\]
which implies that $b(a)=0_{\R^d}$.
\end{proof}
\begin{Aexa}\label{exa.LM}
Let $\ph_1,\ph_2\in C^1_\sharp(Y_1)$ be two $1$-periodic functions on $\R$ such that
\beq\label{phi}
\ph_i(0)=0,\ \ph_i(1/2)=1\;\;\mbox{for }i\in\{1,2\}\quad\mbox{and}\quad \ph_1'(0)\,\ph_2'(1/2)\neq 0,
\eeq
and let $F\in C^1(\R^2)^2$ be the lift of \cite[Example~2]{LlMa} defined by
\beq\label{FLM}
F(x):=x+\big(\ph_2(x_2+\ph_1(x_1)),\ph_1(x_1)\big)\quad \mbox{for }x=(x_1,x_2)\in\R^2,
\eeq
which, due to $F-{\rm Id}\in C^1_\sharp(Y_2)^2$, performs a $C^1$-diffeomorphism on $Y_2$, whose jacobian determinant is equal to $1$.
Let us check that lift $F$ cannot be represented by $F=X(1,\cdot)$ for any flow $X$ solution to ODE~\eqref{bX}.
\par
Set $a:=(0,1/2)$.
First of all, by \eqref{FLM} and \eqref{phi} we have
\[
F(a)-a=e_1\in\Z^2\quad\mbox{and}\quad \nabla F(a)-I_d=\begin{pmatrix}\ph_1'(0)\,\ph_2'(1/2) & \ph_2'(1/2) \\ \ecart \ph_1'(0) & 0\end{pmatrix},
\]
which due to $\ph_1'(0)\,\ph_2'(1/2)\neq 0$ and \eqref{Fa}, implies that $b(a)=0_{\R^2}$.
\par
Now, assume that $F=X(1,\cdot)$ for some flow $X$ associated with $b\in C^1_\sharp(Y_d)^d$ by \eqref{bX}.
Let $\mu\in\cI_b$ be an invariant probability measure on $Y_2$ for the flow $X$.
By Fubini's theorem we have
\beq\label{muFmub}
\ba{ll}
\mu(F-{\rm Id}) & \dis =\int_{Y_2}(F(x)-x)\,d\mu(x)=\int_{Y_2}(X(1,x)-x)\,d\mu(x)
\\ \ecart
& \dis =\int_{Y_2}\kern -.2em\left(\int_0^1 b(X(t,x))\,dt\right)d\mu(x)=\int_0^1\left(\int_{Y_2}b(X(t,x))\,d\mu(x)\right)dt
\\ \ecart
& \dis =\int_0^1\mu(b)\,dt=\mu(b).
\ea
\eeq
This combined with \cite[Definition~$(1.1)$, Theorem~3.4~$(b)$]{MiZi1}, \cite{MiZi2} (in which \cite[Example~2]{LlMa} is revisited) and the equality from \eqref{rhobCb}, yields
\[
\sfC_b=\big\{\mu(b):\mu\in\cI_b\big\}={\rm conv}\,(\rho(b))=\rho(b)=\big\{\mu(F-{\rm Id}):\mu\in\cI_f\big\}=\rho(F)=[0,1]^2,
\]
where $\cI_f$ is the set composed of the probability measures on $Y_2$, which are invariant by the map~$f$ of which $F$ is a lift, {\em i.e.} $f\circ\Pi=\Pi\circ F$ (note that $\cI_b\subset\cI_f$).
However, since $b(a)$ is the null vector, the Dirac measure $\delta_a$ belongs to $\cI_b$.
Hence, applying \eqref{muFmub} with $\mu=\delta_a$, we deduce that
\[
0_{\R^2}=b(a)=\delta_a(b)=\delta_a(F-{\rm Id})=F(a)-a=e_1,
\]
which gives a contradiction.
\end{Aexa}
\subsubsection{Illustration of the collinearity property of the ODE's flow}\label{ss.exacol}
The following example provides a class of two-dimensional vector fields $b$ whose each associated flow admits an explicit infinite family of invariant probability measures satisfying the collinearity property~\eqref{FMcol}.
\begin{Aexa}\label{exa.pre}
Let $b$ be a vector field such that $b=\rho\,R_\perp\nabla u$, where $\rho$ is a non negative function in $C^1_\sharp(Y_2)$ with $\sigma:=1/\rho\in L^1_\sharp(Y_2)$, and where $\nabla u\in C^1_\sharp(Y_2)^2$ with $\overline{\nabla u}\in\Z^2$.
Let $\theta$ be a positive function in $C^0_\sharp(Y_1)$, and let $\Theta$ be its primitive satisfying $\Theta(0)=0$.
\par
First, note that the function $\theta(u):=\theta\circ u$ is $\Z^2$-periodic, since the $\Z^2$-periodicity of $\nabla u$ with $\overline{\nabla u}\in\Z^2$ and the $1$-periodicity of $\theta$ yield
\[
\forall\,k\in\Z^2,\ \forall\,x\in Y_2,\quad \theta(u)(x+k)=\theta\big(u(x)+\overline{\nabla u}\cdot k\big)=\theta(u(x)).
\]
Hence, we deduce that
\[
\ba{lll}
\forall\,k\in\Z^2,\ \forall\,x\in Y_2, & \dis \Theta(u)(x+k) & =\Theta(u(x+k))=\Theta(u(x)+\overline{\nabla u}\cdot k)
\\ \ecart
& & \dis =\int_0^{u(x)+\overline{\nabla u}\cdot k}\theta(t)\,dt
\\ \ecart
& & \dis = \int_0^{\overline{\nabla u}\cdot k}\theta(t)\,dt+\int_{\overline{\nabla u}\cdot k}^{u(x)+\overline{\nabla u}\cdot k}\theta(t)\,dt
\\ \ecart
& & \dis =\bar{\theta}\;\overline{\nabla u}\cdot k+\Theta(u)(x),
\ea
\]
which implies that
\beq\label{Theta}
\left(x\mapsto\Theta(u)(x)-\bar{\theta}\;\overline{\nabla u}\cdot x\right)\in C^1_\sharp(Y_2).
\eeq
Next, by virtue of the divergence-curl result of Proposition~\ref{pro.divcurl} the $\Z^2$-periodic vector-valued function
\beq\label{sithu}
\sigma\,\theta(u)\,b=\sigma\,\rho\,\theta(u)\,R_\perp\nabla u=\theta(u)\,R_\perp\nabla u=R_\perp\nabla(\Theta(u))
\eeq
is divergence free in $\R^2$, which again by Proposition~\ref{pro.divcurl} implies that the probability measure on~$Y_2$  defined by (recall that the function $\sigma\,\theta(u)$ is positive)
\[
d\mu_\theta(x)={\sigma(x)\,\theta(u(x))\over\overline{\sigma\,\theta(u)}}\,dx
\]
is invariant for the flow \eqref{bX} associated with the vector field $b$.
\par\noindent
Therefore, we obtain an infinite family of invariant probability measures $\mu_\theta$ for positive functions $\theta\in C^0_\sharp(Y_1)$, and by \eqref{Theta}, \eqref{sithu} we have
\beq\label{muthb}
\mu_\theta(b)={1\over\overline{\sigma\,\theta(u)}}\,\int_{Y_2}\sigma(x)\,\theta(u(x))\,b(x)\,dx
={1\over\overline{\sigma\,\theta(u)}}\,R_\perp\overline{\nabla\Theta(u)}
={\bar{\theta}\over\overline{\sigma\,\theta(u)}}\,R_\perp\overline{\nabla u}.
\eeq
Hence, all the vectors $\mu_\theta(b)\in\sfC_b$ are parallel to the vector $R_\perp\overline{\nabla u}$ according to assertion~\eqref{FMcol}.
\end{Aexa}
\subsection{The Franks-Misiurewicz result revisited}\label{ss.revFM}
\subsubsection{A partial proof of the collinearity result}\label{ss.FMcol}
In this section we give a partial proof of the collinearity property \eqref{FMcol} based on the decomposition~\eqref{bmuusi} (the proof of which is given in Appendix~\ref{app.rep}) and the divergence-curl result~\eqref{dmub=0}.
Although this first proof is incomplete due to an additional hypothesis, it provides an approach which is both simple and quite different from the purely ergodic approach of \cite{FrMi}.
\begin{apro}\label{pro.FMcol}
Let $b$ a two-dimensional vector field in $C^1_\sharp(Y_2)^2$.
Assume that there exists an invariant probability measure $\mu\in\cI_b$ for the flow $X$ \eqref{bX} satisfying
\beq\label{bmusi}
|b(x)|\,d\mu(x)=\sigma(x)\,dx\;\;\mbox{with}\;\;\sigma\in L^2_\sharp(Y_2)\quad\mbox{and}\quad \mu(b)\neq 0_{\R^2}.
\eeq
Then, the collinearity result \eqref{FMcol} holds.
\end{apro}
\begin{proof}{ of Proposition~\ref{pro.FMcol}}
Let $\mu\in\cI_b$ be an invariant probability measure satisfying \eqref{bmusi}, and let $\nu\in\cI_b$ be any  invariant probability measure for the flow \eqref{bX}.
We will use the following result the proof of which is standard.
\begin{alem}\label{lem.rhon}
Let $\theta$ be a non negative even function in $C^\infty_c(\R^2)$, whose support is contained in the ball $B_{1/2}$ of $\R^2$, centered on $0_{\R^2}$ and of radius $1/2$, and such that $\theta=1$ in the ball~$B_{1/4}$.
Let $\theta_n$ be the function defined by $\theta_n(x):=n/\pi\,e^{-n|x|^2}$ for $x\in\R^2$, $n\in\N$.
Consider the even mollifier function $\rho_n\in C^\infty_\sharp(Y_2)$ defined by the $\Z^2$-periodized function
\[
\rho_n:=\big[\,\theta\,\theta_n\,/\,\overline{\theta\,\theta_n}\,\big]_\sharp\quad\mbox{for }n\in\N.
\]
Then, for any function $f\in L^2_\sharp(Y_2)$, the sequence $\big((d_{Y_2}\,\rho_n)*f\big)_{n\in\N}$ converges uniformly to $0$ in~$Y_2$ (recall the definition \eqref{dY2} of $d_{Y_2}$).
\end{alem}
By the decomposition~\eqref{bmuusi} we have
\beq\label{vns}
\rho_n*(b\,\nu)=\nu(b)+R_\perp \nabla v_n^\sharp\quad\mbox{pointwise in }Y_2,\quad\mbox{with}\quad v^\sharp_n:=\rho_n*v^\sharp.
\eeq
(note that $\overline{\rho_n*(b\,\nu)}=\nu(b)$ due to $\overline{\rho_n}=1$).
Then, using successively the divergence-curl equality~\eqref{dmub=0}, decomposition \eqref{vns} and Fubini's theorem, we get that
\beq\label{rnbmunu}
\ba{ll}
0 & \dis =\int_{Y_2}\nabla v^\sharp_n(x)\cdot b(x)\,\mu(dx)=\int_{Y_2}R_\perp\big(\nu(b)-\rho_n*(b\,\nu)(x)\big)\cdot b(x)\,\mu(dx)
\\ \ecart
& \dis = \det\,(\mu(b),\nu(b))+\int_{Y_2}\kern-.2em\left(\int_{Y_2}\rho_n(x-y)\,R_\perp(b(x)-b(y))\cdot b(x)\,d\nu(y)\right)\mu(dx)
\\ \ecart
& \dis = \det\,(\mu(b),\nu(b))+\int_{Y_2}\kern-.2em\left(\int_{Y_2}\rho_n(y-x)\,R_\perp(b(x)-b(y))\cdot b(x)\,\mu(dx)\right)d\nu(y).
\ea
\eeq
Moreover, applying the mean value theorem to $b$ and using assumption \eqref{bmusi}, we have
\[
\forall\,y\in Y_2,\quad
\left|\,\int_{Y_2}\rho_n(y-x)\,R_\perp(b(x)-b(y))\cdot b(x)\,\mu(dx)\,\right|\leq\|\nabla b\|_{L^\infty(Y_2)^{2\times 2}}\big((\rho_n\,d_{Y_2})*\sigma\big)(y),
\]
which, by virtue of Lemma~\ref{lem.rhon} with $f=\sigma$, converges uniformly to $0$ with respect to $y$ in $Y_2$.
This combined with~\eqref{rnbmunu} implies that $\det\,(\mu(b),\nu(b))=0$ for any $\nu\in\cI_b$.
Finally, from the assumption $\mu(b)\neq 0_{\R^2}$, we deduce the desired collinearity \eqref{FMcol}. 
\end{proof}
\begin{arem}\label{rem.Qcol}
In the above proof of Proposition~\ref{pro.FMcol}, we need one of the two invariant measures to be regular. Otherwise, we have to deal with a delicate problem of product of two measures. Indeed, it is not evident that the regularization of an invariant measure for the flow $X$ provides an invariant measure for $X$. For instance, the regularization by convolution is not effective, since the regularized measure of an invariant measure is not invariant in general.
Hence, the following question occurs naturally:
\beq\label{Qcol}
\mbox{\rm How to prove property~\eqref{FMcol} in the ODE's framework with specific ODE's tools ?}
\eeq
\end{arem}
\subsubsection{Revisiting the Franks-Misiurewicz theorem with the ODE's flow}
First of all, note that if the flow $X$ has an infinite periodic compact orbit $X(\R,x_0)$ in the torus~$Y_2$ for some $x_0\in Y_2$, {\em i.e.} satisfying
\beq\label{infperorb}
\exists\,T>0,\ \exists\,k\in\Z^2\setminus\{0_{\R^2}\},\quad X(T,x_0)=x_0+k,
\eeq
then (recalling Remark~\ref{rem.X1}) we have
\beq\label{kT}
\lim_{t\to\infty}{X(t,x_0)\over t}={k\over T} \in\rho(b)=\sfC_b,
\eeq
since
\[
\forall\,n\in\N,\;\;X(n\,T,x_0)=x_0+n\,k\quad\mbox{and}\quad\lim_{n\to\infty}{X(n\,T,x_0)\over n\,T}={k\over T}.
\]
Note that $k/T$ is then a non null and commensurable vector of $\sfC_b$.
\par\bigskip
Actually, the Franks-Misiurewicz Theorem~\ref{thm.FM} is more accurate in the ODE's context thanks to the following result.
\begin{apro}\label{pro.Cb2d}
Let $b$ be a two-dimensional vector field in $C^1_\sharp(Y_2)^2$.
Then, the Herman rotation set $\sfC_b$ \eqref{Cb} satisfies the following two-by-two disjoint cases:
\begin{itemize}
\item[$({\rm I})$] \vskip -.2cm
$\sfC_b = I \,\zeta$, where $I$ is a closed segment of $\R$ not reduced to a single point, and $\zeta$ is a commensurable vector in $\R^2$.
In particular, this holds if the flow $X$ \eqref{bX} has two infinite periodic compact orbits in $Y_2$ with distinct asymptotics.

\item[$({\rm II})$] \vskip -.2cm
$\sfC_b = \{\zeta\}$, where $\zeta$ is a commensurable vector in $\R^2\setminus\{0_{\R^2}\}$.
In this case, the vector field~$b$ does not vanish in $Y_2$, the flow $X$ has at least one infinite periodic compact orbit in~$Y_2$, and all the infinite periodic orbits have $\zeta$ as asymptotics.

\item[$({\rm III})$] \vskip -.2cm
$\sfC_b = \{\zeta\}$, where $\zeta$ is an incommensurable vector of $\R^2$, if, and only if, $b$ does not vanish in~$Y_2$ and the flow $X$ has no periodic compact orbit in~$Y_2$.
For the {\em if}, we assume in addition that $b\in C^3_\sharp(Y_2)^2$.

\item[$({\rm IV})$] \vskip -.2cm
$\sfC_b$ is either $\{0_{\R^2}\}$, or $\sfC_b = I\,\zeta$, where $I$ is a closed segment of $\R$ such that $0\in I\neq\{0\}$ and $\zeta$ is an incommensurable vector in $\R^2$.
In this case, the vector field $b$ does vanish in $Y_2$, and the flow $X$ has no periodic compact orbit in $Y_2$.
\end{itemize}
\end{apro}
\begin{arem}\label{rem.diff}
Let $\Psi\in C^2(\R^2)^2$ be a $C^2$-diffeomorphism on the torus $Y_2$, {\em i.e.}
\beq\label{psiApsid}
\Psi(y)=A\,y+\Psi_\sharp(y)\quad\mbox{for }y\in \R^2,
\eeq
where $A\in\Z^{2\times 2}$ with $\det(A)=\pm 1$ and $\Psi_\sharp\in C^2_\sharp(Y_2)^2$. Then, by virtue of \cite[Remark~2.1]{BrHe1} we have
\beq\label{asyhX}
\forall\,x\in Y_2,\quad\lim_{t\to\infty}{\widehat{X}(t,x)\over t}\mbox{ exists}\;\;\Leftrightarrow\;\;\lim_{t\to\infty}{X(t,\Psi^{-1}(x))\over t}\mbox{ exists},
\eeq
where $\widehat{X}$ is the ODE's flow defined by
\beq\label{hX}
\widehat{X}(t,x):=\Psi\big(X(t,\Psi^{-1}(x))\big)\quad\mbox{for }(t,x)\in \R\times Y_2,
\eeq
associated with the vector field $\widehat{b}$ defined by
\beq\label{hb}
\widehat{b}(x):=\nabla\Psi(\Psi^{-1}(x))\,b(\Psi^{-1}(x))\quad\mbox{for }x\in Y_2.
\eeq
Moreover, the equivalence~\eqref{asyhX} combined with equality $\sfC_b={\rm conv}\,(\rho(b))$ (recall \eqref{rhob} and \eqref{rhobCb}) implies that
\beq\label{ChbCb}
\sfC_{\widehat{b}}=A\,\sfC_b.
\eeq
Hence, since matrices $A,A^{-1}$ map any integer vector to an integer vector, it is easy to check that each of the four cases of Proposition~\ref{pro.Cb2d} is stable under the change of flow \eqref{hX}. 
\end{arem}
\begin{arem}\label{rem.Qincom}
In view of the incommensurable case $(c)$ of Theorem~\ref{thm.FM} (see Section~\ref{ss.conj})), the only missing result in Proposition~\ref{pro.Cb2d} is that in case $({\rm IV})$ the null vector $0_{\R^2}$ is also an end point of the closed line segment $\sfC_b$.
Actually, the Franks-Misiurewicz incommensurable case, in particular the end of the proof of their \cite[Proposition~1.6]{FrMi}, remains rather mysterious for non experts in ergodic theory.
Indeed, an algebraic condition, {\em i.e.} the existence of an incommensurable vector $\nu(b)$ in any rotation set~$\sfC_b$ not reduced to a single point, does imply the non negativity result
\beq\label{FMnpos}
0_{\R^2}\in \sfC_b\quad\mbox{and}\quad\forall\,\mu(b)\in\sfC_b,\;\; \mu(b)\cdot\nu(b)\geq 0,
\eeq
which is an equivalent way to regard the statement $(c)$ of Theorem~\ref{thm.FM}.
On the other hand, similarly to \eqref{Qcol} one can ask the natural question:
\beq\label{Qincom}
\ba{c}
\mbox{\rm How to prove property~\eqref{FMnpos} in the ODE's framework with specific ODE's tools ?}
\ea
\eeq
We have not managed to answer this question.
However, in Section~\ref{ss.morerotset} and in Section~\ref{ss.Sflow} we present various results and instances which partially explains property~\eqref{FMnpos}, together with the alternative of case $({\rm IV})$ (see Remark~\ref{rem.incom}) and the fact that the case $\sfC_{b}\neq \{0_{\R^2}\}$ in $({\rm IV})$ imposes some sharp restriction on the regularity of the vector field $b$ (see Remark~\ref{rem.Sinc}).
\end{arem}
\begin{arem}\label{rem.Cb2d}
First, note that the collinearity property \eqref{FMcol} is a by-product of Proposition~\ref{pro.Cb2d}.
Classically in ergodic theory (see, {\em e.g.}, the scheme of the proof of \cite[Theorem~3.1]{Tas}) the cases $({\rm I})$-$({\rm II})$ and the case $({\rm III})$ of Proposition~\ref{pro.Cb2d}, when the vector field $b$ does not vanish in $Y_2$, are equivalent respectively to:
\begin{itemize}
\item \vskip-.3cm
the existence of an infinite periodic compact orbit of the flow $X$ in $Y_2$ according to \eqref{infperorb},
\item \vskip-.3cm
the non existence of a periodic compact orbit of the flow $X$ in $Y_2$.
\end{itemize}
\vskip-.3cm
Or equivalently (using {\em e.g.} \cite[Theorem~14.1]{Har}), in terms of the so-called rotation number $\alpha$ of the flow $X$ (see, {\em e.g.}, \cite[Lemma~13.1]{Har}) the alternative can be written respectively as:
\begin{itemize}
\item \vskip-.3cm
$\alpha$ is rational,
\item \vskip-.3cm
$\alpha$ is irrational.
\end{itemize}
Moreover, in case $({\rm III})$ by \cite[Section~4]{Str1} the flow $X$ is uniquely ergodic, and by \cite[Section~5]{Oxt1} the limit of $X(t,x)/t$ as $t\to\infty$ to $\zeta$ is uniform with respect to $x\in Y_2$.
\par
In the proof below, we will combine these two approaches. While in the ergodic literature the use of the rotation number is quite classical, the use of the asymptotics of the flow is more unusual. In view of the homogenization of the transport equation, Tassa \cite[Theorems~4.2,~4.5]{Tas} distinguished two asymptotic behaviors based on the rationality or not of the rotation number, but assuming the existence of an invariant probability measure for the flow with a regular density.
More generally, Peirone \cite[Theorem~3.1]{Tas} specified the asymptotics of the flow under the sole assumption that the vector field $b$ is non vanishing in $Y_2$, but he did not study the commensurability of the asymptotics.
Actually, the link between the asymptotics of the flow and the rotation number is not immediate.
\end{arem}
\par\noindent
\begin{proof}{ of Proposition~\ref{pro.Cb2d}}
Let us first prove the following result which is based on the first case of the proof of \cite[Theorem~3.1]{Pei}.
\begin{alem} \label{lem-per}
Assume that the flow $X$ \eqref{bX} has an infinite periodic compact orbit $X(\R,x_0)$ in $Y_2$ for some $x_0\in Y_2$, {\em i.e.} satisfying~\eqref{infperorb}. Then, the rotation set $\sfC_b$ contains the non null vector $\zeta := k/T$, and $\sfC_b=I\,\zeta$ for some closed segment $I$ of $\R$ with $I\neq \{0\}$.   
\end{alem}
\begin{proof}{ of Lemma~\ref{lem-per}}
Under the assumptions of Lemma~\ref{lem-per} Peirone proved that for any orbit $X(\R,x)$, $x\in\R^2$, there exist $\alpha,\beta\in\R$ such that
\beq\label{comorb}
\forall\,t\in\R,\quad X(t,x)\cdot k^\perp\in(\alpha,\beta),\quad\mbox{with }k^\perp:=R_\perp k.
\eeq
On the other hand, let $\mu_e$ be an ergodic invariant probability measure for the flow $X$ \eqref{bX}.
By virtue of Birkhoff's theorem there exists a vector $\zeta_e\in\sfC_b$ such that
\[
\lim_{t\to\infty}{X(t,x)\over t}=\zeta_e\quad\mbox{for $\mu$-a.e. }x\in Y_2.
\]
This combined with \eqref{comorb} implies that $\zeta_e\cdot k^\perp=0$, {\em i.e.} $\zeta_e\parallel k$.
Hence, since by \eqref{CbEb} $\sfC_b$ is the compact convex hull of the vectors $\zeta_e$ for $\mu_e\in\cE_b$, we deduce that $\sfC_b=J\,k$ for some closed segment $J$ of $\R$, with $1/T\in J$ due to \eqref{kT}. 
\end{proof}
\par\medskip\noindent
{\it Proof of the sufficient condition of $({\rm I})$.} Assume that the flow $X$ \eqref{bX} has two infinite periodic compact orbits in $Y_2$ with distinct asymptotics $k_1/T_1$ and $k_2/T_2$ for some $k_1,k_2\in\Z^2$.
Hence, we deduce from Lemma~\ref{lem-per} that $\sfC_b=I\,\zeta$ for some commensurable vector $\zeta$ of $\R^2$, and for some closed segment $I$ of $\R$ which is not reduced to a single point, since $k_1/T_1$ and $k_2/T_2$ are in $\sfC_b$.  
\par\medskip\noindent
{\it Proof of $({\rm III})$.}
If $\sfC_b = \{\zeta\}$ with $\zeta$ an incommensurable vector of $\R^2$, then $b$ does not vanish in $Y_2$ due to $\zeta\neq 0_{\R^2}$.
This corresponds exactly to the second case in the proof of \cite[Theorem~3.1]{Pei}, {\em i.e.} the flow $X$ \eqref{bX} has no infinite periodic compact orbit in $Y_2$.
\par
Conversely, assume that the vector field $b$ does not vanish in $Y_2$, and that the flow $X$ has no periodic compact orbit in $Y_2$.
In this case, Peirone \cite[Theorem~3.1]{Pei} proved that there exists a vector $\zeta\in\R^2$ such that
\beq\label{zeta}
\forall\,x\in Y_2,\quad\lim_{t\to\infty}{X(t,x)\over t}=\zeta,
\eeq
or equivalently, by \cite[Proposition~2.1]{BrHe1} $\sfC_b=\{\zeta\}$.
\par\noindent
It remains to show that the vector $\zeta$ is incommensurable (see Remark~\ref{rem.Cb2d}).
To this end, assuming in addition that $b\in C^3_\sharp(Y_2)^2$, we will proceed to a constructive proof based on two classical homeomorphic flows (see \cite[Chapter~VII.14]{Har}).
\par
First, by \cite[Section~3]{Aku} (see also \cite[Theorem~2]{Str2}) combined with the regularity $b\in C^3_\sharp(Y_2)^2$, there exists a $C^2$-diffeomorphism $\Psi$ on $Y_2$ such that the flow $\widehat{X}$ defined by \eqref{hX} is associated with the vector field $\widehat{b}\in C^1_\sharp(Y_2)^2$ defined by \eqref{hb} satisfying $\widehat{b}_1>0$ in $Y_2$.
By Remark~\ref{rem.diff} the diffeomorphism $\Psi$ reads as \eqref{psiApsid}, where $A\in\Z^{2\times 2}$ with $\det\,(A)=\pm1$ and $\Psi_\sharp\in C^1_\sharp(Y_2)^{2\times 2}$.
Hence, we easily deduce that
\beq\label{Chb}
\forall\,x\in Y_2,\quad\lim_{t\to\infty}{\widehat{X}(t,x)\over t}=A\,\zeta.
\eeq
Moreover, since
\[
\forall\,x\in Y_2,\ \forall\,k\in\Z^2,\quad\Psi^{-1}(x+k)-\Psi^{-1}(x)=A^{-1}k\in\Z^2,
\]
and the flow $X$ has no compact orbit, the flow $\widehat{X}$ defined by \eqref{hX} has no compact orbit either.
\par
Next, define the mapping $\widetilde{X}\in C^1(\R;C^1_\sharp(Y_2))^2$ by
\beq\label{tX}
\widetilde{X}(t,x):=\widehat{X}(f(t,x),x)\quad\mbox{for }(t,x)\in\R\times Y_2,
\eeq
where $f(\cdot,x)$ is the solution to the first-order ODE
\beq\label{ftX}
{\partial f\over\partial t}(t,x)={1\over \widehat{b}_1\big(\widehat{X}(f(t,x),x)\big)},\quad f(0,x)=0.
\eeq
Let $x\in\R^2$ and $k\in\Z^2$. Since the function $\widehat{b}_1(\widehat{X}(t,\cdot))$ is $\Z^2$-periodic for any $t\in\R$, $f(\cdot,x)$ and $f(\cdot,x+k)$ are both solutions to equation \eqref{ftX}. Hence, by the uniqueness of such a solution we deduce that $f(\cdot,x)$ and $f(\cdot,x+k)$ agree in $\R$, which implies that $f(t,\cdot)$ is $\Z^2$-periodic for any $t\in\R$.
Moreover, ${\partial f/\partial t}$ is uniformly bounded from above and below by positive constants in $\R\times Y_2$, so is $f(t,\cdot)/t$ in $(0,\infty)\times Y_2$ by the mean value theorem.
\par\noindent
Next, by the chain rule it is easy to check that the mapping $\widetilde{X}$ agrees with the flow associated with the vector field $\widetilde{b}\in C^1_\sharp(Y_2)^2$ defined by
\beq\label{tb}
\widetilde{b}:=\big(1,\widehat{b}_2/\widehat{b}_1\big)\quad\mbox{in }Y_2.
\eeq
Moreover, by the formula (14.4) of \cite[Exercise~14.3, p.~199]{Har} we have
\beq\label{Ctb}
 \lim_{t\to\infty}{\widetilde{X}(t,x)\over t}=(1,\alpha),
\eeq
where $\alpha$ is the Poincar\'e rotation number of the flow $\widetilde{X}$ defined by (see, {\em e.g.}, \cite[Lemma~13.1]{Har})
\beq\label{alpha}
\alpha:=\lim_{n\to\infty}{\widetilde{X}_2(n,x_2e_2)\over n}\quad\mbox{for any }x_2\in\R.
\eeq
Since $f(\cdot,x)$ is positive in $(0,\infty)$ and the flow $\widehat{X}$ has no compact orbit, neither has the flow $\widetilde{X}$ defined by \eqref{tX}.
Hence, by virtue of \cite[Theorem~14.1]{Har} the number $\al$ is irrational.
\par
Finally, from \eqref{zeta}, \eqref{Chb}, \eqref{hX}, \eqref{Ctb}, \eqref{alpha} we deduce that
\beq\label{zeal}
\forall\,x\in Y_2,\quad \lim_{t\to\infty}{\widetilde{X}(t,x)\over t}=\left(\lim_{t\to\infty}{f(t,x)\over t}\right)A\,\zeta=(1,\alpha).
\eeq
Indeed, the boundedness of $f(t,\cdot)/t$ from above and below by positive constants in $(0,\infty)\times Y_2$ implies that $A\,\zeta$ is not null, which in return implies that the limit of $f(t,x)/t$ as $t\to\infty$ does exist.
Therefore, due to the invertibility of $A$ in $\Z^{2\times 2}$ and the irrationality of $\alpha$, equality \eqref{zeal} yields the incommensurability of $\zeta$. 
\par\medskip\noindent
{\it Proof of the necessary condition for $({\rm II})$.}
Assume that $\sfC_b = \{\zeta\}$ with a non null commensurable vector $\zeta$ of $\R^2$.
On the one hand, due to $\zeta\neq 0_{\R^2}$ the vector field $b$ does not vanish in $Y_2$.
On the other hand, it follows from the equivalence of case $({\rm III})$ that the flow $X$ \eqref{bX} has at least one infinite periodic compact orbit in $Y_2$. Finally, we deduce from the sufficient condition of case~$({\rm I})$ that all infinite periodic compact orbits have $\zeta$ as asymptotics.  
\par\medskip\noindent
{\it Proof of  $({\rm IV})$.}
Assume that both cases $({\rm I})$, $({\rm II})$, $({\rm III})$ do not hold.
First of all, note that if the vector field $b$ does not vanish in $Y_2$, then the alternative in the proof of \cite[Theorem~3.1]{Pei} shows that either case $({\rm I})$ or case $({\rm II})$ is satisfied if there exists an infinite periodic orbit in $Y_2$, or case $({\rm III})$ is fulfilled if there is no periodic orbit in $Y_2$.
Therefore, in the present case the vector field $b$ does vanish in $Y_2$, and thus $0_{\R_2}\in\sfC_b$.
\par
Now, assume that there exist two non collinear vectors $\xi$ and $\eta$ in $\sfC_b=\rho(b)$ (by Remark~\ref{rem.X1}).
Then, by convexity the rotation set $\sfC_b$ contains the closed triangle $\dT$ of vertices $0_{\R^2}$, $\xi$, $\eta$, which has a non empty interior $\mathring{\dT}$. Hence, by virtue of \cite[Theorem~3.2]{Fra2} the flow $X$ has two periodic orbits \eqref{infperorb} with distinct asymptotics $k_1/T_1$ and $k_2/T_2$ in $(\mathring{\dT}\cap \Q^2)\setminus\{0_{\R^2}\}$. These orbits are infinite periodic compact orbits with distinct asymptotics, which implies that $\sfC_b$ satisfies case $({\rm I})$, and leads us to a contradiction.
Therefore, the rotation set $\sfC_b$ is a closed line segment passing through $0_{\R^2}$, {\em i.e.} $\sfC_b = I \,\zeta$ with $I$ a closed segment of~$\R$ containing~$0$, and $\zeta$ a non null vector of $\R^2$.
If $I=\{0\}$, then we have $\sfC_b = \{0_{\R^2}\}$.
Otherwise, if $I\neq \{0\}$, since case $({\rm I})$ does not hold by hypothesis, the vector $\zeta$ is necessarily incommensurable.
\par
The proof of Proposition~\ref{pro.Cb2d} is now complete.
\end{proof}
\subsection{New results on Herman's rotation set}\label{ss.morerotset}
In this section we study several extensions of the Franks-Misiurewicz Theorem~\ref{thm.FM}, which are specific to the ODE's context.
\par
First, the Herman rotation set $\sfC_b$ can be actually characterized more precisely in the commensurable cases $({\rm I})$ and $({\rm II})$ of Proposition~\ref{pro.Cb2d}, when $b$ does not vanish on $Y_2$.
\begin{apro}\label{pro.Cb2dcom}
Let $b$ be a two-dimensional vector field in $C^1_\sharp(Y_2)^2$ which does not vanish in~$Y_2$, and such that the flow $X$ \eqref{bX} has at least one infinite periodic compact orbit in $Y_2$. Let us write $b:=\rho\,\Phi$, where $\rho$ is a positive function in $C^1_\sharp(Y_2)$ and $\Phi$ is a non vanishing vector field in~$C^1_\sharp(Y_2)^2$. Then, we have the following results:
\begin{itemize}
\item[$(i)$] There exist a unique commensurable unit vector $\zeta_\Phi$ of $\R^2$, only depending on~$\Phi$, satisfying either $\zeta_\Phi\cdot e_1>0$ or $\zeta_\Phi=e_2$, and a closed segment $I_b$ of $\R$ not reduced to $\{0\}$, such that the rotation set reads as $\sfC_b=I_b\,\zeta_\Phi$.
Moreover, if $b$ does not vanish in $Y_2$ and $\sfC_b$ is not a unit set, $I_b$ is a segment of the type:
\beq\label{Ibne0}
I_b=[\al,\be]\quad\mbox{with}\quad 0<\al<\be\;\;\mbox{or}\;\;\al<0<\be,
\eeq
{\em i.e.} $0_{\R^2}$ is not an end point of the rotation set $\sfC_b$.
\item[$(ii)$] The commensurable case $({\rm I})$ of Proposition~\ref{pro.Cb2d} is stable under a uniform perturbation of the positive function~$\rho$ with a fixed vector field $\Phi$ in the representation $b=\rho\,\Phi$.
\end{itemize}
\end{apro}
\begin{proof}{ of Proposition~\ref{pro.Cb2dcom}}
\par\smallskip\noindent
{\it Proof of part $(i)$.}
By Lemma~\ref{lem-per} the rotation set reads as $\sfC_b=J\,k$ for some closed segment $J\neq \{0\}$ of $\R$, and some vector $k\in\Z^2\setminus\{0_{\R^2}\}$. Now, consider the flow $X_\Phi$ associated with the vector field $\Phi$ (recall that $b=\rho\,\Phi$).
Let $\nu$ be an invariant probability measure for the flow $X_\Phi$, and define the probability measure $\mu$ on $Y_2$ by
\[
d\mu(x):={c_\nu\over\rho(x)}\,d\nu(x)\quad\mbox{where}\quad c_\nu:=\left(\int_{Y_2}{1\over\rho(x)}\,d\nu(x)\right)^{-1}\in(0,\infty).
\]
By virtue of Proposition~\ref{pro.divcurl} we have
\[
\forall\,\varphi\in C^1_\sharp(Y_d),\quad \int_{Y_2} b(x)\cdot\nabla\varphi(x)\,d\mu(x)=c_\nu\int_{Y_d} \Phi(x)\cdot\nabla\varphi(x)\,d\nu(x)=0,
\]
which in return implies that the probability measure $\mu$ is invariant for the flow $X$ associated with $b$.
Hence, we get that
\beq\label{CPhiCb}
\int_{Y_2} \Phi(x)\,d\nu(x)=1/c_\nu\int_{Y_2} b(x)\,d\mu(x)\in 1/c_\nu\,\sfC_b.
\eeq
Similarly, for any invariant probability measure $\mu$ for the flow $X$, there exist an invariant probability measure $\nu$ for the flow $X_\Phi$ and a constant $c_\mu>0$ such that
\beq\label{CbCPhi}
\int_{Y_2} b(x)\,d\mu(x)=1/c_\mu\int_{Y_2} \Phi(x)\,d\mu(x)\in 1/c_\mu\,\sfC_\Phi.
\eeq
Therefore, combining \eqref{CPhiCb} and \eqref{CbCPhi} with $\sfC_b=I\,k$, the rotation set $\sfC_\Phi$ associated with $\Phi$ reads as $\sfC_\Phi=J\,k$ for some closed segment $J$ of $\R$ not reduced to $\{0\}$.
Hence, the compact convex set $\sfC_\Phi$ uniquely reads as $\sfC_\Phi=I_\Phi\,\zeta_\Phi$, where $\zeta_\Phi$ is the unit vector parallel to $k$ satisfying either $\zeta_\Phi\cdot e_1>0$ or $\zeta_\Phi=e_2$, and $I_\Phi$ is a closed segment of $\R$ not reduced to $\{0\}$.
Therefore, the vector $\zeta_\Phi$ is commensurable in $\R^2$, and the rotation set can be written as $\sfC_b=I\,k=I_b\,\zeta_\Phi$ for some closed segment $I_b$ of $\R$ not reduced to $\{0\}$.
\par
Now, assume that the vector field $b$ does not vanish in $Y_2$, and that $\sfC_b$ is not a unit set.
Then, the segment $I_b$ does satisfy \eqref{Ibne0}.
Otherwise, the null vector $0_{\R^2}$ is an end point of the closed line segment $\sfC_b=I_b\,\zeta_\Phi$.
Hence, by representation \eqref{CbEb} there exists an ergodic invariant probability measure $\mu$ such that $0_{\R^2}=\mu(b)$.
Therefore, by virtue of the Franks result \cite[Theorem~3.5]{Fra1} the vector field $b$ does vanish, which yields a contradiction.
\par\medskip\noindent
{\it Proof of part $(ii)$.}
Let $\rho_0$ be a positive function in $C^1_\sharp(Y_2)$ and let $\Phi_0$ be a non vanishing vector field in $C^1_\sharp(Y_2)^2$.
Set
\[
m_0:=\min_{x\in Y_2} \rho_0(x)>0.
\]
If the vector field $b_0 = \rho_0\,\Phi_0$ satisfies the conditions of the case~(I) of Proposition~\ref{pro.Cb2d}, it is enough to prove the existence of $r\in (0,m_0)$ such that any vector field $b = \rho\,\Phi_0$ with $|\rho-\rho_0| < r$, also fulfills the conditions of case~(I). Note that the assumptions of the case~(I) for a non vanishing vector field $b$ in $Y_2$ are equivalent to the following conditions stated by Peirone \cite[Theorem~31]{Pei}: the flow associated with $b$ admits at least one infinite compact orbit and $\# \sfC_b > 1$.
\par
To prove the desired stability statement, consider the vector field $b_0$ as above, and assume by contradiction that there exist a decreasing sequence $(r_n)_{n\geq 1}$ of positive numbers satisfying $r_1 < m_0$ and converging to $0$, and a sequence $(\theta_n)_{n\geq 1}$ of positive functions in $C^1_\sharp(Y_2)$  satisfying
\[
\forall\,n\geq 1,\quad |\theta_n-\rho_0| < r_n,
\]
such that for any $n\geq 1$, the vector field $b_n = \theta_n\,\Phi_0$ does not satisfy the conditions of case~(I).
This combined with the case $(2)$ of \cite[Theorem~31]{Pei} (noting that $b_n$ does not vanish in $Y_2$) implies that $\# \sfC_{b_n} = 1$. Then, an easy adaptation of the perturbation result \cite[Theorem~3.1]{BrHe1} (using that the sequence $(\rho/\theta_n)_{n\geq 1}$ is uniformly bounded) shows that $\# \sfC_{b_0} = 1$, which yields a contradiction.
\par
This concludes the proof of of Proposition~\ref{pro.Cb2dcom}.
\end{proof}
\par\bigskip
The next result provides an incomplete, but a rather large framework which illustrates the incommensurable case $(c)$ of Theorem~\ref{thm.FM}, and which makes valid some appropriate ODE's approach.
It deals with two general classes of vector fields for which the incommensurable case~(IV) of Proposition~\ref{pro.Cb2d} holds true.
The first class is a simple illustration of case (IV).
Roughly speaking, the second class is based on any vanishing vector field $b$ associated with an invariant probability measure $\sigma(x)\,dx$, such that $\sigma b$ is a non vanishing regular vector field. Note that the Lebesgue density $\sigma$ has to blow up at the roots of $b$ and to be regular outside, in order to the vector field $\sigma b$ to be both non vanishing and regular.
\begin{apro}\label{pro.FMincom}
Let $b$ be a vector field in $C^1_\sharp(Y_2)^2$.
We have the following results:
\begin{itemize}
\item[$(i)$] Assume that there exists an integer vector $k\in\Z^2\setminus\{{0_{\R^2}}\}$ satisfying $k\cdot b\geq 0$ in $Y_2$, and that $\sfC_b=I\,\zeta$ with $I$ a closed segment of $\R$ and $\zeta$ an incommensurable vector of $\R^2$.
Then, we have $I\subset[0,\infty)$ or $I\subset(-\infty,0]$.
If in addition $0_{\R^2}\in \sfC_b$, then the null vector $0_{\R^2}$ is an end point of the closed line segment $\sfC_b$.
\item[$(ii)$] Assume that $b$ does vanish in $Y_2$, and that there exists an invariant probability measure $d\mu(x)=\sigma(x)\,dx\in\cI_b$ with a non negative function $\sigma\in L^1_\sharp(Y_2)$ of mean value $1$, such that $\sigma\,b$ is a non vanishing vector field in $C^1_\sharp(Y_2)^2$ and $\mu(b)$ is incommensurable in $\R^2$.
Then, the rotation set $\sfC_b$ is a closed line segment of $\R^2$ not reduced to a single point, with one end at $0_{\R^2}$ and having irrational slope.
\end{itemize}
\end{apro}
\begin{proof}{ of Proposition~\ref{pro.FMincom}}
\par\smallskip\noindent
{\it Proof of part $(i)$.} We have for any invariant probability measure $\mu\in\cI_b$,
\[
k\cdot\mu(b)=\int_{Y_2}\underbrace{k\cdot b(x)}_{\geq 0}d\mu(x)\geq 0,
\]
which implies that
\[
k\cdot\sfC_b=(k\cdot\zeta)\,I\subset[0,\infty).
\]
Moreover, since the vector $\zeta$ is incommensurable and $k\neq 0_{\R^2}$, we have $k\cdot\zeta\neq 0$.
Hence, we deduce that the segment $I$ satisfies $I\subset[0,\infty)$ or $I\subset(-\infty,0]$.
If in addition $0_{\R^2}\in \sfC_b$, then there exists $\al\in\R$ such $I=[0,\al]$ or $I=[\al,0]$, so that $0_{\R^2}$ is an end point of the closed line segment $\sfC_b=I\,\zeta$.
\par\medskip\noindent
{\it Proof of part $(ii)$.}
First note that our assumption is equivalent to the existence of a gradient field $\nabla u\in C^1_\sharp(Y_2)^2$ such that
\[
\sigma\,b=R_\perp\nabla u\;\;\mbox{a.e. in }Y_2\quad\mbox{and}\quad
\overline{\nabla u}\;\;\mbox{is incommensurable in }\R^2,
\]
since
\[
\overline{\nabla u}=-\int_{Y_2}R_\perp b(x)\,d\mu(x)=-\,R_\perp\,\mu(b)\;\;\mbox{is incommensurable in }\R^2.
\]
Hence, by virtue of the \cite[Corollary~3.4~(3.22)]{BrHe1} applied with $\rho=1/\sigma$ and $a=1$, the rotation set $\sfC_b$ is given by
\[
\sfC_b=[0_{\R^2},\zeta]\quad\mbox{with}\quad \zeta:=\underline{\rho}\,R_\perp\overline{\nabla u}=R_\perp\overline{\nabla u}\neq 0_{\R^2}.
\]
Therefore, $\sfC_b$ is a closed line segment with one end at $0_{\R^2}$ and having irrational slope due to the incommensurability of $\zeta$ in $\R^2$.
\end{proof}
\par\bigskip
The incommensurable case~(IV) of Proposition~\ref{pro.Cb2d} is by far the most intricate.
In contrast with Proposition~\ref{pro.FMincom} $(ii)$, the following general result shows that the non singleton case in the alternative of~(IV) is actually exceptional due to a counterintuitive non negativity constraint.
\begin{atheo}\label{thm.Ssgn}
Let $b=a\,\Phi$ be a vector field in $C^1_\sharp(Y_2)^2$, where $a$ is a changing sign function in $C^1_\sharp(Y_2)$, and where $\Phi$ is a non vanishing vector field in  $C^3_\sharp(Y_2)^2$.
Also assume that there exists a positive function $\sigma$ in $C^3_\sharp(Y_2)$ with mean value $1$, such that $\sigma\,\Phi$ is divergence free in $\R^2$, and such that the mean value $\overline{\sigma\,\Phi}$ is incommensurable in $\R^2$.
Then, the support of any invariant probability measure in $\cI_b$ is contained in the set $\{a=0\}$, and thus in particular $\sfC_b=\{0_{\R^2}\}$.
\end{atheo}
\begin{arem}\label{rem.incom} 
Theorem~\ref{thm.Ssgn} shows that the change of sign of the function $a$ forces the Herman rotation set $\sfC_{a\,\Phi}$ to be the unit set $\{0_{\R^2}\}$ in the incommensurable case.
On the contrary, when $a$ is non negative (or non positive), the part~$(ii)$ of Proposition~\ref{pro.FMincom} (with $\sigma=1/a$ and $\sigma\,b=\Phi$) shows that the rotation set $\sfC_{a\,\Phi}$  is a closed line segment with the null vector $0_{\R^2}$ as an end point.
Roughly speaking, these two results are the two complementary sides of the incommensurable case, and illuminate the point $(c)$ of the Franks-Misiurewicz Theorem~\ref{thm.FM} in the ODE's context.
\end{arem}
\begin{proof}{ of Theorem~\ref{thm.Ssgn}}
\par\smallskip\noindent
{\it First step:} The case of the Stepanoff flow, {\em i.e.} the vector field $\Phi$ is a constant vector $\zeta\in\R^2$.
\par\noindent
By virtue of the ergodic decomposition theorem (see, {\em e.g.}, \cite[Theorem~14.2]{Cou}) it is enough to prove that
\beq\label{mua0=1}
\forall\,\mu\in\cE_{a\,\zeta},\quad \mu(\{a=0\})=1.
\eeq
Assume by contradiction that there exists an ergodic invariant probability measure $\mu\in\cE_{a\,\zeta}$ such that $\mu(\{a=0\})<1$.
\par
Note that the sets $\{\pm\,a>0\}$ are invariant for the flow $X$ associated with $a\,\zeta$.
Indeed, for any $x\in Y_2$ such that $\pm\,a(x)>0$, we have
\[
\forall\,t\in\R,\quad \pm\,a(X(t,x))>0.
\] 
Otherwise, there exists $s\in\R$ such that $a(X(s,x))=0$, which implies that $X(\R,x)=\{x\}$, and thus $a(x)=0$, a contradiction. 
\par
Then, by ergodicity we have $\mu(\{a>0\})=1$ or $\mu(\{a<0\})=1$. Without loss of generality we can assume that $\mu(\{a>0\})=1$.
Next, by Proposition~\ref{pro.divcurl} the product of the non negative Radon measure $\nu$ on $Y_2$ defined by
\[
d\nu(x):=a(x)\,d\mu(x)=a^+(x)\,d\mu(x)\quad (\mbox{where $a^+$ denotes the non negative part of $a$})
\]
by the incommensurable vector $\zeta$ is divergence free in $\cD'(\R^2)$.
Thus, by virtue of the ergodic \cite[Lemma~5.2]{BrHe2} there exists a  non negative constant $c$ such that $d\nu(x)=c\,dx$.
Hence, we deduce that
\[
\nu(\{a<0\})=\int_{\{a<0\}}a^+(x)\,d\mu(x)=0=c\int_{\{a<0\}}\,dx.
\]
However, since the continuous function $a$ does change sign in $Y_2$, we have
\[
\int_{\{a<0\}}\,dx>0.
\]
Therefore, we get that $c=0$ and $\nu=0$, which implies that
\[
\mu(\{a>0\})=\int_{\{a>0\}}a^+(x)\,d\mu(x)=\int_{\{a>0\}}d\nu(x)=0,
\]
a contradiction with $\mu(\{a>0\})=1$.
This establishes \eqref{mua0=1}.
\par\medskip\noindent
{\it Second step:} The case where the vector field $\Phi$ is not constant.
\par\noindent
Using successively \cite[Section~3]{Aku} (which leads us to the case $\Phi_1>0$ and which needs the regularity $b\in C^3_\sharp(Y_2)^2$) and the extension of a Kolmogorov's theorem \cite[Theorem~2.3]{Tas} (which is based on the hypothesis $\Phi_1>0$) together with the assumption that $\sigma\,\Phi$ is a non vanishing divergence free vector field in $\R^2$, we get that the flow $X_\Phi$ associated with the vector field $\Phi$ is homeomorphic to a Stepanoff flow $X_{\widehat{\Phi}}$ associated with some vector field $\widehat{\Phi}=\widehat{\alpha}\,\zeta$.
More precisely, there exists a diffeomorphism $\Psi\in C^2_\sharp(Y_2)^2$ on the torus $Y_2$, satisfying \eqref{psiApsid} with $A\in\Z^{2\times 2}$ and $\det(A)=\pm 1$, such that equality \eqref{hX} holds true for $X_\Phi$ and \eqref{hb} reads as
\beq\label{hal}
\forall\,x\in Y_2,\quad \widehat{\alpha}(x)\,\zeta=\nabla\Psi(\Psi^{-1}(x))\,\Phi(\Psi^{-1}(x)).
\eeq

\par
First of all, since the probability measure $\sigma(y)\,dy$ is invariant for the flow $X_\Phi$, by \eqref{ChbCb} and the equality $\widehat{\Phi}=\widehat{\alpha}\,\zeta$, there exist a non empty compact interval $I$ of $\R$ such that
\[
\overline{\sigma\,\Phi}\in\sfC_\Phi=A^{-1}\,\sfC_{\widehat{\Phi}}=A^{-1}\,(I\,\zeta)=I\,(A^{-1}\zeta),
\]
which, due to $A^{-1}\in\Z^{2\times 2}$, implies that $\zeta$ is incommensurable in $\R^2$.
\par
On the other hand, in view of \eqref{hal} it is easy to see that the flow $\widehat{X}$ defined by
\beq\label{hXPsi}
\widehat{X}(t,x):=\Psi\big(X(t,\Psi^{-1}(x))\big)\quad\mbox{for }(t,x)\in\R\times Y_2,
\eeq
is the flow associated with the vector field
\beq\label{hbha}
\widehat{b}=\widehat{a}\,\zeta\quad\mbox{where}\quad \widehat{a}(x):=a(\Psi^{-1}(x))\,\widehat{\alpha}(x)\;\;\mbox{for }x\in Y_2.
\eeq
Now, let $\mu$ be a probability measure on $Y_2$ and let $\widehat{\mu}$ be the pushforward measure of $\mu$ by $\Psi$.
Then, by equality \eqref{hXPsi} we have for any function $f\in L^1_\sharp(Y_2,\mu)$ and $\widehat{f}:=f\circ\Psi^{-1}$,
\beq\label{fhf}
\forall\,t\in\R,\quad
\left\{\ba{rl}
\dis \int_{Y_2} f(X(t,y))\,d\mu(y) & \dis = \int_{Y_2} \widehat{f}(\widehat{X}(t,x))\,d\widehat{\mu}(x)
\\ \ecart
\dis \int_{Y_2} \big|f(X(t,y))-f(y)\big|\,d\mu(y) & \dis = \int_{Y_2} \big|\widehat{f}(\widehat{X}(t,x))-\widehat{f}(x)\big|\,d\widehat{\mu}(x).
\ea\right.
\eeq
The first equality of \eqref{fhf} shows that the probability measure $\mu$ is invariant for the flow $X$ if, and only if, the probability measure $\widehat{\mu}$ is invariant for the flow $\widehat{X}$.
The second equality of \eqref{fhf} shows that the probability measure $\mu$ is ergodic for the flow $X$ (recall definition \eqref{Xerg}) if, and only if, the probability measure $\widehat{\mu}$ is ergodic for the flow $\widehat{X}$.
Therefore, since by \eqref{hal} the function $\widehat{\alpha}$ does not vanish in $Y_2$, from the first case applied with the Stepanoff flow $\widehat{X}$ we deduce that for any invariant probability measure $\mu$ in $\cI_{a\,\Phi}$,
\[
\mu(\{a=0\})=\widehat{\mu}\left(\big\{a\circ\Psi^{-1}=0\big\}\right)=\widehat{\mu}\left(\big\{(a\circ\Psi^{-1})\,\widehat{\alpha}=0\big\}\right)
=\widehat{\mu}(\{\widehat{a}=0\})=1,
\]
which concludes the proof of Theorem~\ref{thm.Ssgn}.
\end{proof}
\subsection{A complete picture of the Stepanoff flow}\label{ss.Sflow}
It turns out that the Stepanoff flows allow us to illustrate the four cases of Proposition~\ref{pro.Cb2d} as well as the results of Section~\ref{ss.morerotset}.
To this end, consider a vector field $b=a\,\zeta$ in $C^1_\sharp(Y_2)^2$, where $a$ is a function in $C^1_\sharp(Y_2)$ and $\zeta$ is a unit vector in $\R^2$.
\begin{enumerate}
\item 
Assume that the function $a$ has a constant sign in $Y_2$, for instance $a\geq 0$ in $Y_2$.
\par\noindent
On the one hand, when the vector $\zeta$ is commensurable in $\R^2$, {\em i.e.} $T\,\zeta\in\Z^2$ for some $T>0$, we have the following alternative:
\begin{itemize}
\item If the function $a(\cdot\,\zeta+x)$ does not vanish in $\R$, it is easy to check that
\beq\label{Fx}
X(t,x)=F_x^{-1}(t)\,\zeta+x,\quad \mbox{where}\quad F_x(t):=\int_0^t {ds\over a(s\,\zeta+x)}\quad \mbox{for }t\in\R,
\eeq
and $F_x^{-1}$ denotes the reciprocal of the function $F_x$.
Therefore, we get that
\beq\label{asyXa>0}
\lim_{t\to\infty} \frac{X(t,x)}{t}=m(x)\,\zeta,\quad\mbox{where}\quad m(x):=\left({1\over T}\int_0^T{dt\over a(t\,\zeta+x)}\right)^{-1}\in(0,\infty).
\eeq
The limit $m(x)$ in \eqref{asyXa>0} is deduced from the $T$-periodicity of $a(\cdot\,\zeta+x)$ (due to $T\,\zeta\in\Z^2$ and the $\Z^2$-periodicity of $a$).
\item If $a(x)=0$, then we have $X(\cdot,x)=x$. Otherwise, if $a(x)\neq 0$ and the function $a(\cdot\,\zeta+x)$ does vanish in $\R$, then the new function $F_x$ obtained by a similar expression as the one of \eqref{Fx}, is defined not on $\R$ but on the bounded interval $(\al_x,\be_x)$ of $\R$, where $\al_x<0$, respectively $\be_x>0$, is the largest negative, respectively the smallest positive, root of the $T$-periodic function $a(\cdot\,\zeta+x)$.
More precisely, we have
\beq\label{Fxze}
\left\{\ba{lcl}
& \dis X(t,x)=F_x^{-1}(t)\,\zeta+x & \mbox{for }t\in\R,
\\ \ecart
\mbox{where} & \dis F_x(u):=\int_0^u {ds\over a(s\,\zeta+x)} & \mbox{for }u\in(\al_x,\be_x).
\ea\right.
\eeq
Hence, the reciprocal $F_x^{-1}$ maps $\R$ on the bounded interval $(\al_x,\be_x)$, so that the trajectory $X(\cdot,x)$ is bounded in $\R^2$.
Therefore, we get that
\beq\label{asyXa=0}
\dis \lim_{t\to\infty} \frac{X(t,x)}{t}=0_{\R^2}.
\eeq
\end{itemize}
\noindent
On the other hand, when the vector $\zeta$ is incommensurable in $\R^2$, by virtue of \cite[Theorem~3.1]{BrHe1} (see also \cite[Proposition~5.4]{BrHe2}) we have
\beq\label{CbSincom}
\sfC_{a\,\zeta}:=\left\{\ba{rl}
\{\underline{a}\}\,\zeta& \mbox{if $a>0$ in $Y_2$}
\\ \ecart
[0,\underline{a}]\,\zeta & \mbox{if $a$ does vanish in $Y_2$},
\ea\right.
\quad\mbox{where}\quad \underline{a}:=\left(\int_{Y_2}{dy\over a(y)}\right)^{-1}\in[0,\infty).
\eeq
More precisely, when $a>0$ in $Y_2$, the unit set $\{\underline{a}\}\,\zeta$ is also deduced from formula~\eqref{Fx}.
Then, we approximate the continuous function $1/a$ uniformly on~$Y_2$ by Fej\'er's type trigonometric polynomials, and we notice that
\[
\forall\,k\in \Z^2\setminus\{0_{\R^2}\},\quad \lim_{t\to \infty}\left({1\over t}\int_0^t e^{-2i\pi\,s\,\zeta\cdot k}\,ds\right)=0.
\]
On the contrary, when the function $a$ does vanish in $Y_2$, we may apply the perturbation result \cite[Theorem~3.1]{BrHe1} with the sequence $a_n:=a+1/n>0$ for $n\geq 1$, since by the previous case $\sfC_{a_n\,\zeta}=\{\underline{a_n}\}\,\zeta$ and by Beppo-Levi's theorem
\[
\lim_{n\to\infty}\,\underline{a_n}=\lim_{n\to\infty}\left(\int_{Y_2}{dy\over a(y)+1/n}\right)^{-1}=\underline{a}\in[0,\infty).
\]
\par\medskip
Therefore, collecting limits \eqref{asyXa>0}, \eqref{asyXa=0} combined with \eqref{rhob}, \eqref{rhobCb}, and the incommensurable case \eqref{CbSincom}, we obtain the complete characterization of the Herman rotation set $\sfC_b$ for the Stepanoff flow, when the function $a$ has a constant sign: 
\beq\label{CbS}
\sfC_{a\,\zeta}=\left\{\ba{rl}
\dis \big[\min_{Y_2}m,\max_{Y_2}m\big]\,\zeta & \mbox{if $\zeta$ is commensurable in $\R^2$}
\\
\{\underline{a}\}\,\zeta & \mbox{if $\zeta$ is incommensurable and $a>0$ in $Y_2$}
\\ \ecart
[0,\underline{a}]\,\zeta & \mbox{if $\zeta$ is incommensurable and $a$ does vanish in $Y_2$},
\ea\right.
\eeq
with the convention
\beq\label{acon}
\left\{\ba{rl}
\underline{a}:=0 & \mbox{if }1/a\notin L^1_\sharp(Y_2)
\\ \ecart
m(x):=0 & \mbox{if $a(\cdot\,\zeta+x)$ does vanish in }\R.
\ea\right.
\eeq
Note that in the first case of \eqref{CbS}, the whole rotation set $\sfC_b$ may be represented thanks to the asymptotics of the flow $X$, namely we have
\beq\label{repCbcom}
\forall\,\mu\in\cI_b,\ \exists\,x\in Y_2,\quad \mu(b)=\lim_{t\to\infty}\,{X(t,x)\over t}.
\eeq
The situation is less clear in the third case of \eqref{CbS}.
\item Now, assume that the function $a$ does change sign in $Y_2$, which in particular implies that $a$ does vanish in $Y_2$ (with possibly an infinite number of roots).
\par
The case where $\zeta$ is commensurable is quite similar.
Indeed, using the explicit formula \eqref{Fxze} of the flow $X$, the first formula of \eqref{CbS} for $\sfC_b$ still holds.
\par
On the contrary, applying Theorem~\ref{thm.Ssgn} with $\Phi=\zeta$, the case where $\zeta$ is incommensurable leads us to the null asymptotics of the flow, or equivalently by \eqref{Cbsindis}, to $\sfC_b=\{0_{\R^2}\}$.
\end{enumerate}
\par\medskip\noindent
Let us now illustrate the cases $({\rm I})$, $({\rm II})$ $({\rm III})$, $({\rm IV})$ of Proposition~\ref{pro.Cb2d} thanks to Stepanoff's flows with suitable functions $a$.
\noindent
\begin{itemize}
\item[$({\rm I})$] Let $\zeta:=e_1$, and define $a(x):=a_1(x_1)\,a_2(x_2)$ for $x=(x_1,x_2)\in Y_2$, where $a_1$ is a positive function in $C^1_\sharp(Y_1)$ and $a_2$ is a non constant function in $C^1_\sharp(Y_1)$ with a finite positive number of roots in $Y_1$ (so that $a$ has a constant sign).
Then, the function $m$ defined in \eqref{asyXa>0} is given by
\[
m(x)=a_2(x_2)\left(\int_0^1{dt\over a_1(t+x_1)}\,dt\right)^{-1}=\underline{a_1}\,a_2(x_2)\quad\mbox{for }x\in Y_2,
\]
where $\underline{a_1}$ denotes the harmonic mean of $a_1$ on $Y_1$.
Therefore, from the first case of \eqref{CbS} we deduce that
\[
\sfC_{a\,e_1}=\big[\min_{Y_1}a_2,\max_{Y_1}a_2\big]\,\underline{a_1}\,e_1,\quad\mbox{with}\quad \min_{Y_1}a_2<\max_{Y_1}a_2.
\]
\item[$({\rm II})$] We take again the previous example, but assuming this time that the function $a_2$ is a constant function $c_2\in\R$.
Therefore, we deduce that
\[
\sfC_{a\,e_1}=\{\underline{a_1}\,c_2\}\,e_1.
\]
\item[$({\rm III})$] Let $\zeta$ be an incommensurable vector in $\R^2$, and let $a$ be any positive (or negative) function in $C^1_\sharp(Y_2)$.
Therefore, from the second case of \eqref{CbS} we deduce that
\[
\sfC_{a\,e_1}=\{\underline{a}\}\,\zeta.
\]
\item[$({\rm IV})$] Let $\zeta$ be an incommensurable vector in $\R^2$, and let $a$ be any vanishing non negative (or non positive) function in  $C^1_\sharp(Y_2)$ with a finite number of roots in $Y_2$.
Therefore, from the third case of \eqref{CbS} we deduce that that $\sfC_{a\,\zeta}$ is the closed line segment
\[
\sfC_{a\,\zeta}=[0,\underline{a}]\,\zeta.
\]
An example of such a function $a$ is given by
\[
a(x):=\big(\sin^2(\pi x_1)+\sin^2(\pi x_2)\big)^\alpha,\quad x=(x_1,x_2)\in Y_2,\quad\mbox{for any }\al\in(1/2,\infty),
\]
which satisfies
\[
\underline{a}=\left\{\ba{cl}
\in(0,\infty) & \mbox{if }\al\in(1/2,1)
\\
0 & \mbox{if }\al\geq 1.
\ea\right.
\]
\end{itemize}
\par\bigskip
Let us conclude this section by a remark on the necessary regularity restriction of the vector field $b$ under the incommensurable case $(c)$ of the Franks-Misiurewicz Theorem~\ref{thm.FM}.
\begin{arem}\label{rem.Sinc}
Theorem~\ref{thm.Ssgn} shows that the incommensurable case of the Stepanoff flow associated with a vector field $a\,\zeta$, with $a\in C^1_\sharp(Y_2)$ and $\zeta$ incommensurable in $\R^2$, is fully characterized by the four following situations:
\beq\label{CbSinc}
\sfC_{a\,\zeta}=\left\{\ba{cl}
\{0_{\R^2}\} & \mbox{if $a$ changes sign in $Y_2$}
\\ \ecart
\{0_{\R^2}\} & \mbox{if $a$ has a constant sign and does vanish in $Y_2$, with }1/a\notin L^1_\sharp(Y_2)
\\ \ecart
[0,\underline{a}]\,\zeta & \mbox{if $a$ has a constant sign and does vanish in $Y_2$, with }1/a\in L^1_\sharp(Y_2)
\\ \ecart
\{\underline{a}\}\,\zeta & \mbox{if $a$ does not vanish in $Y_2$}.
\ea\right.
\eeq
Therefore, the case $(c)$ of Theorem~\ref{thm.FM} holds only when $a$ does vanish and has a constant sign in $Y_2$ with $1/a\in L^1_\sharp(Y_2)$.
This condition is rather restrictive, since it cannot be satisfied by a function $a\in C^2_\sharp(Y_2)$.
Indeed, if for instance the non negative function $a\in C^2_\sharp(Y_2)$ vanishes at the point $0_{\R^2}$, then we have $\nabla a(0_{\R^2})=0_{\R^2}$, and we deduce that there exist $R>0$ and $C>0$ satisfying
\beq\label{1/a2d}
\int_{Y_2}{dx\over |a(x)|}\geq \int_{\{|x|<R\}}{dx\over |a(x)|}\geq \int_{\{|x|<R\}}{C\over |x|^2}\,dx= 2\pi\,C\int_0^R{dr\over r}=\infty.
\eeq
Beyond the Stepanoff flows class, the same restriction applies in the more general framework of Proposition~\ref{pro.FMincom} $(ii)$.
Indeed, in this setting the vector field $b$ reads as $b=a\,\Phi$, where $\Phi$ is a non vanishing vector field, and where $a=1/\sigma$ with $\sigma(x)\,dx$ an invariant probability measure in $\cI_b$, is a vanishing non negative function which cannot belong to $C^2_\sharp(Y_2)$ due to \eqref{1/a2d}.
\end{arem}
%%%%%%%%%%
\section{Fourier relations satisfied by invariant measures}\label{s.intFrel}
The following result provides an integral relation and equivalent Fourier relations satisfied by any pair of invariant probability measures for the flow \eqref{bX}. As a by product the Fourier relations \eqref{Fmubnub} for $j=k$ extend the Franks-Misiurewicz collinearity result~\eqref{FMcol}. 
\begin{atheo}\label{thm.relmunu}
Let $\mu$ and $\nu$ be two invariant probability measures for the flow \eqref{bX} associated with a vector field $b\in C^1_\sharp(Y_2)^2$.
Then, there are unique (up to additive constants) stream functions $u^\sharp,v^\sharp$ in $BV_\sharp(Y_2)$ defined by the vector-valued measure representations
(see Remark~\ref{rem.mubDdu} just below)
\beq\label{bmuus}
b\,\mu=\mu(b)+R_\perp\nabla u^\sharp\quad\mbox{and}\quad b\,\nu=\nu(b)+R_\perp\nabla v^\sharp\quad\mbox{in }Y_2.
\eeq
Moreover, for any function $\rho\in C^2_\sharp(Y_2\!\times\!Y_2)$, we have the following integral relation
\beq\label{rhomunu}
\ba{l}
\dis \int_{Y_2}\kern -.2em\int_{Y_2}\rho(x,y)\det\,(b(x),b(y))\,d\mu(x)\,d\nu(y)
\\ \ecart
\dis =R_\perp\nu(b)\cdot\int_{Y_2}\kern -.2em\left(\int_{Y_2}\rho(x,y)\,dy\right)b(x)\,d\mu(x)
-R_\perp\mu(b)\cdot\int_{Y_2}\kern -.2em\left(\int_{Y_2}\rho(x,y)\,dx\right)b(y)\,d\nu(y)
\\ \ecart
\dis +\int_{Y_2}\kern -.2em\int_{Y_2}\kern -.2em\left({\partial^2\rho\over\partial x_1\partial y_2}-{\partial^2\rho\over\partial x_2\partial y_1}\right)u^\sharp(x)\,v^\sharp(y)\,dxdy,
\ea
\eeq
or equivalently, in terms of the Fourier coefficients
\beq\label{Fmubnub}
\forall\,(j,k)\in(\Z^2\setminus\{0_{\R^2}\})^2\cup\{(0_{\R^2},0_{\R^2})\},\quad \det\big(\widehat{\mu b}(j),\widehat{\nu b}(k)\big)
=-\,4\pi^2\det\,(j,k)\,\widehat{u^\sharp}(j)\,\widehat{v^\sharp}(k).
\eeq
\end{atheo}
\begin{arem}\label{rem.mubDdu} 
The two equalities of \eqref{bmuus} have to be understood in the sense of the vector-valued Radon measures on $Y_2$, including an integration by parts for the orthogonal gradient term.
For instance, the first equality of \eqref{bmuus}:
\begin{itemize}
\item on the space $\R^2$, according to the definition of the Radon measure $\widetilde{\mu}\in\cM(\R^2)$ in~\eqref{tmu}, means that for any smooth function with compact support  $\Phi$ in $C^\infty_c(\R^2)^2$,
\beq\label{btmuDdu}
\dis \int_{\R^2}b(x)\cdot\Phi(x)\,\widetilde{\mu}(x)=\int_{\R^2}\mu(b)\cdot\Phi(x)\,dx+\int_{\R^2}{\rm div}(R_\perp\Phi)(x)\,u^\sharp(x)\,dx,
\eeq
\item or equivalently, on the torus $Y_2$ using \eqref{tmu}, means that for any $\Z^2$-periodic function $\Psi$ in $C^\infty_\sharp(Y_2)^2$, which is associated with a periodized function $\Psi:=\Phi_\sharp$ by \cite[Lemma~3.5]{Bri1},
\beq\label{bmuDdu}
\ba{ll}
\dis \int_{Y_2}b(x)\cdot\Psi(x)\,d\mu(x) & \dis =\int_{Y_2}\mu(b)\cdot\Psi(x)\,dx-\int_{Y_2}R_\perp\Psi(x)\cdot d\nabla u^\sharp(x)
\\ \ecart
& \dis =\int_{Y_2}\mu(b)\cdot\Psi(x)\,dx+\int_{Y_2}{\rm div}(R_\perp\Psi)(x)\,u^\sharp(x)\,dx.
\ea
\eeq
\end{itemize}
See Appendix~\ref{app.rep} below for further details.
\end{arem}
\begin{arem}\label{rem.intrel}
As an immediate consequence of relation \eqref{rhomunu}, we have
\beq\label{fpmmubnub}
\forall\,f\in C^0_\sharp(Y_2),\quad \int_{Y_2}\kern -.2em\int_{Y_2}f(x\pm y)\,\det\,(b(x),b(y))\,d\mu(x)\,d\nu(y)=0.
\eeq
Indeed, for any function $f\in C^2_\sharp(Y_2)$, by \eqref{FMcol} the function $\rho:(x,y)\mapsto f(x\pm y)$ satisfies the equalities
\[
\ba{l}
\dis R_\perp\nu(b)\cdot\int_{Y_2}\kern -.2em\left(\int_{Y_2}f(x\pm y)\,dy\right)b(x)\,d\mu(x)
-R_\perp\mu(b)\cdot\int_{Y_2}\kern -.2em\left(\int_{Y_2}f(x\pm y)\,dx\right)b(y)\,d\nu(y)
\\ \ecart
\dis =2\,\overline{f}\,\det\,(\mu(b),\nu(b))=0,
\ea
\]
and
\[
{\partial^2\rho\over\partial x_1\partial y_2}-{\partial^2\rho\over\partial x_2\partial y_1}=0\quad\mbox{in }Y_2\!\times\!Y_2.
\]
The case where the function $f$ is only continuous easily follows from a density argument.
\end{arem}
\begin{arem}\label{rem.Freld}
Actually, in any dimension $d\geq 2$ and for any invariant probability measure $\mu$ related to the flow~\eqref{bX} induced by any vector field $b\in C^1_\sharp(Y_d)^d$, we have the Fourier relations
\beq\label{Freld}
\forall\,k\in\Z^d\setminus\{0_{\R^d}\},\quad \widehat{\mu b}(k)\cdot k=0.
\eeq
Indeed, applying the divergence-curl equality \eqref{dmub=0} with measure $\mu$ and any regular function $\psi$ in $C^\infty_\sharp(Y_d)$, and using Parseval's equality, we get that
\[
0=\int_{Y_d} \nabla\psi(y)\cdot b(y)\,d\mu(y)=-\,2i\pi\sum_{k\in\Z^d}\widehat{\psi}(k)\,k\cdot \widehat{\mu b}(k),
\]
which, due to the arbitrariness of $\psi$, thus yields \eqref{Freld}.
\end{arem}
\begin{arem}\label{rem.Fcoef}
Note that the case $j=k=0_{\R^2}$ in relation \eqref{Fmubnub} corresponds to the collinearity result~\eqref{FMcol}.
Furthermore, changing $k$ by $-\,k$ for any integer vector~$k$, a by-product of relations \eqref{Fmubnub} is given by the following correlation between the Fourier coefficients of any two divergence free vector-valued measures $\mu\,b$ and $\nu\,b$,
\beq\label{Fjkcol}
\ba{c}
\forall\,j,k\in\Z^2\setminus\{0_{\R^2}\},
\\ \ecart
\det\,(j,k)=0\;\Rightarrow\;
\left\{\ba{l}\dis \det\left(\Re\big(\widehat{\mu b}(j)\big),\Re\big(\widehat{\nu b}(k)\big)\right)
=\det\left(\Im\big(\widehat{\mu b}(j)\big),\Im\big(\widehat{\nu b}(k)\big)\right)=0
\\ \ecart
\dis \det\left(\Re\big(\widehat{\mu b}(j)\big),\Im\big(\widehat{\nu b}(k)\big)\right)
=\det\left(\Im\big(\widehat{\mu b}(j)\big),\Re\big(\widehat{\nu b}(k)\big)\right)=0.
\ea\right.
\ea
\eeq
Therefore, Fourier relations \eqref{Fmubnub} and \eqref{Fjkcol} may be regarded as an extension of the Franks-Misiurewicz \cite[Theorem~1.2]{FrMi} for the continuous flows to the ODE's flows, with more substantial correlation between the invariant measures.
\par
On the other hand, since the stream functions $u^\sharp$ and $v^\sharp$ of \eqref{bmuus} belong to $L^2_\sharp(Y_2)$, from relations \eqref{Fmubnub} and Parseval's formula we also deduce the estimate
\[
\sum_{(j,k)\in \Z^2\times\Z^2,\,\det\,(j,k)\neq 0}\,\left|\,{\det\big(\widehat{\mu b}(j),\widehat{\nu b}(k)\big)\over \det\,(j,k)}\,\right|^2
\leq 16\,\pi^4\,\|u^\sharp\|^2_{L^2_\sharp(Y_2)}\, \|v^\sharp\|^2_{L^2_\sharp(Y_2)}<\infty.
\]
\end{arem}
\begin{proof}{ of Theorem~\ref{thm.relmunu}}
The proof of the classical representation \eqref{bmuus} is postponed in Appendix~\ref{app.rep}.
\par
The case $j=k=0_{\R^2}$ in \eqref{Fmubnub} corresponds exactly to the collinearity result \eqref{FMcol}.
\par\noindent
Now, let $j,k\in\Z^2\setminus\{0_{\R^2}\}$.
Taking the Fourier coefficients of the representations \eqref{bmuus} of the vector-valued measures $\mu\,b,\nu\,b$ respectively at the integer points $j,k$, and using Kronecker's symbol $\delta_{j,k}$, we get that
\[
\ba{ll}
\det\big(\widehat{\mu b}(j),\widehat{\nu b}(k)\big) &
=\big(\mu(b)\,\delta_{j,0_{\R^2}},+R_\perp\widehat{\nabla u^\sharp}(j)\big)\cdot R_\perp\big(\nu(b)\,\delta_{k,0_{\R^2}}+R_\perp\widehat{\nabla v^\sharp}(k)\big)
\\ \ecart
& \dis =\widehat{\nabla u^\sharp}(j)\cdot R_\perp\widehat{\nabla v^\sharp}(k)=-\,4\pi^2\det\,(j,k)\,\widehat{u^\sharp}(j)\,\widehat{v^\sharp}(k),
\ea
\]
which implies the Fourier relations \eqref{Fmubnub}.
\par
On the other hand, a trigonometric polynomial $\Sigma$ in $Y_2\!\times\!Y_2$ can be written as
\[
\Sigma(x,y)=\sum_{(j,k)\in J\!\times\!K}\kern -.2em c_{j,k}\,e^{-2i\pi\,(j\cdot x+k\cdot y)}\quad\mbox{for }(x,y)\in Y_2\!\times\!Y_2,
\]
where $J,K$ are any non empty finite subsets of $\Z^2$, and $c_{j,k}$ for $(j,k)\in J\!\times\!K$ are any complex numbers.
Taking into account equality \eqref{FMcol}, any trigonometric polynomial $\Sigma$ satisfies the three following equalities:
\[
\dis \int_{Y_2}\kern -.2em\int_{Y_2}\Sigma(x,y)\det\,(b(x),b(y))\,d\mu(x)\,d\nu(y)
=\sum_{(j,k)\in J\times K\setminus\{(0_{\R^2},0_{\R^2})\}}\kern -.2em c_{j,k}\,\det\big(\widehat{\mu b}(j),\widehat{\nu b}(k)\big),
\]
and
\[
\ba{l}
\dis -\,R_\perp\mu(b)\cdot\int_{Y_2}\kern -.2em\left(\int_{Y_2}\Sigma(x,y)\,dx\right)b(y)\,d\nu(y)
+R_\perp\nu(b)\cdot\int_{Y_2}\kern -.2em\left(\int_{Y_2}\Sigma(x,y)\,dy\right)b(x)\,d\mu(x)
\\ \ecart
\dis =\sum_{k\in K\setminus\{0_{\R^2}\}}c_{0_{\R^2},k}\det\big(\widehat{\mu b}(0_{\R^2}),\widehat{\nu b}(k)\big)
+\sum_{j\in J\setminus\{0_{\R^2}\}}c_{j,0_{\R^2}}\det\big(\widehat{\mu b}(j),\widehat{\nu b}(0_{\R^2})\big),
\ea
\]
and
\[
\dis \int_{Y_2}\kern -.2em\int_{Y_2}\left({\partial^2\Sigma\over\partial x_1\partial y_2}(x,y)-{\partial^2\Sigma\over\partial x_2\partial y_1}(x,y)\right)u^\sharp(x)\,v^\sharp(y)\,dxdy
=-\,4\pi^2\kern -.2em\sum_{(j,k)\in J\!\times\!K}\kern -.2em c_{j,k}\,\det\,(j,k)\,\widehat{u^\sharp}(j)\,\widehat{v^\sharp}(k).
\]
These equalities combined with relations \eqref{Fmubnub} yield
\beq\label{Sigma}
\ba{l}
\ba{ll}
\dis \int_{Y_2}\kern -.2em\int_{Y_2}\Sigma(x,y)\det\,(b(x),b(y))\,d\mu(x)\,d\nu(y)
& \kern -.6em \dis +\,R_\perp\mu(b)\cdot\int_{Y_2}\kern -.2em\left(\int_{Y_2}\Sigma(x,y)\,dx\right)b(y)\,d\nu(y)
\\ \ecart
& \kern -.6em \dis -\,R_\perp\nu(b)\cdot\int_{Y_2}\kern -.2em\left(\int_{Y_2}\Sigma(x,y)\,dy\right)b(x)\,d\mu(x)
\ea
\\ \ecart
\dis =\sum_{(j,k)\in J\setminus\{0_{\R^2}\}\times K\setminus\{0_{\R^2}\}}\kern -.2em c_{j,k}\,\det\big(\widehat{\mu b}(j),\widehat{\nu b}(k)\big)
\\ \ecart
\dis =-\,4\pi^2\kern -.2em\sum_{(j,k)\in J\setminus\{0_{\R^2}\}\times K\setminus\{0_{\R^2}\}}\kern -.2em c_{j,k}\,\det\,(j,k)\,\widehat{u^\sharp}(j)\,\widehat{v^\sharp}(k)
\\ \ecart
\dis =-\,4\pi^2\kern -.2em\sum_{(j,k)\in J\times K}\kern -.2em c_{j,k}\,\det\,(j,k)\,\widehat{u^\sharp}(j)\,\widehat{v^\sharp}(k)
\\ \ecart
\dis =\int_{Y_2}\kern -.2em\int_{Y_2}\left({\partial^2\Sigma\over\partial x_1\partial y_2}(x,y)-{\partial^2\Sigma\over\partial x_2\partial y_1}(x,y)\right)u^\sharp(x)\,v^\sharp(y)\,dxdy.
\ea
\eeq
Hence, any trigonometric polynomial $\Sigma$ satisfies relation \eqref{rhomunu}.
Therefore, applying Fej\'er's approximation theorem to the continuous functions $\rho$ and ${\partial^2\rho\over\partial x_1\partial y_2}\!-\!{\partial^2\rho\over\partial x_2\partial y_1}$ in $Y_2\!\times\!Y_2$, we deduce from \eqref{Sigma} that any function $\rho\in C^2(Y_2\!\times\!Y_2)$ satisfies the integral relation \eqref{rhomunu}.
\par
The proof of Theorem~\ref{thm.relmunu} is now complete.
\end{proof}
\par\bigskip
The Fourier relations \eqref{Fmubnub} are not satisfied in general for any pair of vectors $(j,k)\in\Z^2$.
Surprisingly, the validity of the Fourier relations \eqref{Fmubnub} adding all the pairs $(j,0_{\R^2})$ and $(0_{\R^2},k)$ with $j,k\in\Z^2\setminus\{0_{\R^2}\}$, characterizes  a subclass of two-dimensional Stepanoff flows as shows the following result.
\begin{apro}\label{pro.0hmub}
Let $b\in C^1_\sharp(Y_2)^2$ be a vector field which has a finite number of roots in $Y_2$, with $1/|b|\in L^1_\sharp(Y_2)$, and such that there exists an invariant probability measure $m\in\cI_b$ for the flow \eqref{bX}, with $m(b)$ incommensurable in $\R^2$. Then, we have the following equivalence
\beq\label{eqFmubnub}
\ba{c}
\forall\,\mu,\nu\in\cI_b,\ \forall\,j,k\in\Z^2,\quad \det\big(\widehat{\mu b}(j),\widehat{\nu b}(k)\big)=-\,4\pi^2\det\,(j,k)\,\widehat{u^\sharp}(j)\,\widehat{v^\sharp}(k)
\\ \ecart
\Updownarrow
\\ \ecart
\dis \exists\,a\in C^1_\sharp(Y_2)\mbox{ with constant sign such that}\quad {1\over a}\in L^1_\sharp(Y_2)\quad\mbox{and}\quad b=a\,m(b)\mbox{ in }Y_2.
\ea
\eeq
\end{apro}
\begin{proof}{ of Proposition~\ref{pro.0hmub}}
Assume that the Fourier relations of \eqref{eqFmubnub} hold.
Let $u^\sharp$ be the stream function in $L^2_\sharp(Y_2)$ associated with $m$, {\em i.e.} by \eqref{bmuus} satisfying
\beq\label{mdu}
m\,b=m(b)+R_\perp\nabla u^\sharp\quad\mbox{in }Y_2.
\eeq
Taking $\mu=\nu=m$ and $j=0_{\R^2}$ in the first assertion of \eqref{eqFmubnub}, we get that
\[
\forall\,k\in\Z^2\setminus\{0_{\R^2}\},\quad \det\big(m(b),\widehat{m b}(k)\big)=-\,2i\pi\,\widehat{u^\sharp}(k)\,m(b)\cdot R_\perp k=0.
\]
This combined with the incommensurability of $m(b)$ yields
\[
\forall\,k\in\Z^2\setminus\{0_{\R^2}\},\quad \widehat{u^\sharp}(k)=0,
\]
which implies that $u^\sharp$ is constant in $Y_2$.
Hence, we deduce from \eqref{mdu}  that $m\,b=m(b)$ in $Y_2$.
Then, we have for any function $\ph\in C^0_\sharp(Y_2)$ and for any $\ep>0$,
\[
\int_{Y_2}\left({\ph(x)\,b(x)\over |b(x)|^2+\ep}\right)\cdot b(x)\,dm(x)=\int_{Y_2}\left({\ph(x)\,b(x)\over |b(x)|^2+\ep}\right)\cdot m(b)\,dx.
\]
Passing to the limit as $\ep\to 0$ in the previous equality and using Lebesgue's theorem with respect to measures $dm(x)$ and $dx$ (recall that $1/|b|\in L^1_\sharp(Y_2)$),
we get that for any $\ph\in C^0_\sharp(Y_2)$,
\[
\int_{Y_2}\ph(x)\,\cha_{\{b\neq 0_{\R^2}\}}(x)\,dm(x)=\int_{Y_2}\ph(x)\,{b(x)\cdot m(b)\over |b(x)|^2}\,\cha_{\{b\neq 0_{\R^2}\}}(x)\,dx,
\]
or equivalently,
\[
\cha_{\{b\neq 0_{\R^2}\}}(x)\,dm(x)={b(x)\cdot m(b)\over |b(x)|^2}\,\cha_{\{b\neq 0_{\R^2}\}}(x)\,dx\quad\mbox{in }Y_2.
\]
Multiplying by $b(x)$ the previous equality and using that $b(x)\,dm(x)=m(b)\,dx$ in $Y_2$, it follows that
\beq\label{bmb}
m(b)={b(x)\cdot m(b)\over |b(x)|^2}\,b(x)\quad\mbox{for any }x\in\{b\neq 0_{\R^2}\}.
\eeq
Therefore, since by hypothesis $b$ has a number finite of roots in $Y_2$ and $m(b)\neq 0_{\R^2}$, we deduce from \eqref{bmb} that the function $a$ defined by
\[
a(x):=\left\{\ba{cl}
\dis {|b(x)|^2\over b(x)\cdot m(b)} & \mbox{if }b(x)\neq 0_{\R^2}
\\
0 &  \mbox{if }b(x)=0_{\R^2},
\ea\right.
\]
satisfies $b = a\,m(b)$ in $Y_2$, belongs to $C^1_\sharp(Y_2)$ since $b$ does, has a constant sign, and
\[
{1\over |a|}\leq {|m(b)|\over |b|}\in L^1_\sharp(Y_2).
\]
\par
Conversely, assume that the vector field $b$ reads as $b=a\,\xi$, where $a$ is a non negative (for instance) function in $C^1_\sharp(Y_2)$ having a finite number of roots in $Y_2$ with $1/a\in L^1_\sharp(Y_2)$, and where $\xi$ is an incommensurable vector in $\R^2$.
Then, by virtue of \cite{Oxt1,Mar} and more precisely by \cite[Proposition~5.1, Remark~5.2]{BrHe2}), the set $\cE_b$ of the ergodic measure of the associated so-called Stepanoff flow~\cite{Ste} is finite and is given by
\[
\cE_b=\big\{m(dx)=\underline{a}/a(x)\,dx\big\}\cup\big\{\delta_x:a(x)=0\big\},
\]
where $m$ is also the unique invariant probability measure on $Y_2$ which does not load the zero set of $a$.
Since by \eqref{IbEb} each probability measure in $\cI_b$ is a convex combination of elements of the finite set $\cE_b$, it is enough to show that any pair $(\mu,\nu)$ of measures in $\cE_b$ satisfies the Fourier relations \eqref{eqFmubnub}.
It is immediate if at least one of the two measures in the pair $(\mu,\nu)$, for instance $\mu$, is a Dirac measure $\delta_x$ with $a(x)=0$.
Indeed, we then have $\delta_x\,b=0_{\R^2}$, so that $\widehat{\mu\,b}=0$.
Otherwise, we have $\mu=\nu=m$.
Therefore, since $m\,b=\underline{a}\,\xi=m(b)$, the associated stream function $u^\sharp$ defined by \eqref{mdu} is constant in $Y_2$, so that the Fourier relations satisfy
\[
\ba{lrll}
\forall\,j,k\in\Z^2,\quad & \dis \det\big(\widehat{mb}(j),\widehat{mb}(k)\big) & = \det\big(m(b)\,\delta_{j,0_{\R^2}},m(b)\,\delta_{k,0_{\R^2}}\big) & =0
\\ \ecart
& -\,4\pi^2\det\,(j,k)\,\widehat{u^\sharp}(j)\,\widehat{u^\sharp}(k) & =-\,4\pi^2\det\,(j,k)\,\delta_{j,0_{\R^2}}\,\delta_{k,0_{\R^2}}\,(\overline{u^\sharp})^2 & =0,
\ea
\]
which yields the first assertion of \eqref{eqFmubnub}. This concludes the proof of Proposition~\ref{eqFmubnub}.
\end{proof}
%%%%%%%%%%
\newpage
\section{Dimension two {\em versus} dimension three}\label{s.3d}
The situation for the three-dimensional ODE's flow is radically different, since we have the following result.
\begin{atheo}\label{thm.Cb3d}
For any closed convex polyhedron $\P$ of $\R^3$ with rational vertices, there exists a vector field $b\in C^1_\sharp(Y_3)^3$ such that the associated Herman rotation set $\sfC_b$ agrees with $\P$.
\end{atheo}
\begin{arem}\label{rem.K}
Theorem~\ref{thm.Cb3d} illuminates the gap between dimension three and dimension two in relation to the asymptotics of the flow. Indeed, the Franks-Misiurewicz Theorem~\ref{thm.FM} ensures that the two-dimensional rotation set $\sfC_b$ is always a closed line segment of $\R^2$.
\par
The three-dimensional result of Theorem~\ref{thm.Cb3d} appears as the extension in dimension three of the similar two-dimensional closure result of Kwapisz~\cite{Kwa} for any convex polygon with rational vertices, which is obtained for a suitable lift $F:\R^2\to\R^2$ of some homeomorphism on $Y_2$ homotopic to identity, rather than the ODE's lift $X(1,\cdot):\R^2\to\R^2$ associated with a suitable two-dimensional vector field $b$.
\par
The proof of Theorem~\ref{thm.Cb3d} is based on a cylinders structure (see Figure~\ref{fig1}) similar to the one used by Llibre and Mackay \cite[Example~4]{LlMa}. However, these authors built a particular flow $F$ on the torus $Y_3$ homotopic to identity, directly from the cylinders structure. In contrast, we construct below a vector field $b$ directly from a general cylinders structure (see from \eqref{CiCj} to \eqref{bphipsi1}), which itself induces the flow $X$ \eqref{bX}. The aim in both approaches consists in deriving a rotation set which has the shape of a convex polyhedron of $\R^3$.
\end{arem}
\begin{figure}[!t]
\centering
\includegraphics[scale=.6]{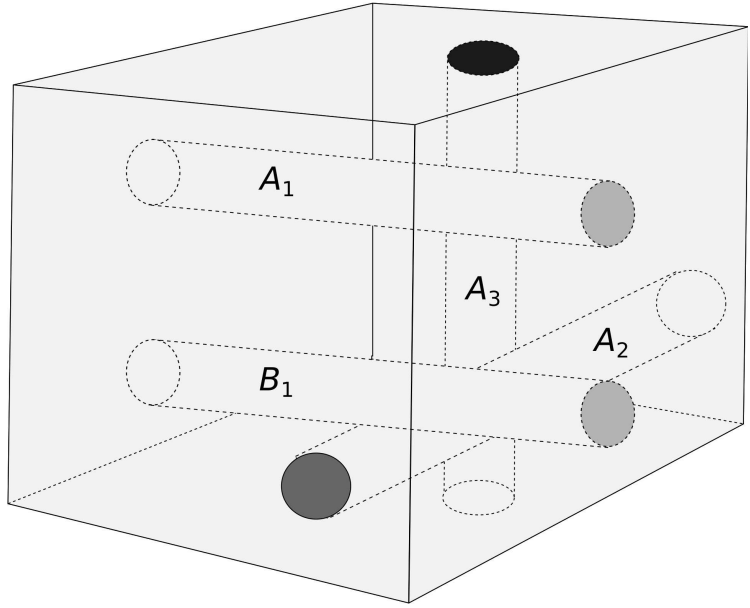}
\caption{\it Four cylinders being at a positive distance from each other in $Y_3$.}
\label{fig1}
\end{figure}
\par
We need the following algebraic result which is proved at the end of the section.
\begin{alem}\label{lem.xiixi}
Let $(\xi^1,\dots,\xi^{n-1})$ be $n-1$ vectors in $\Q^3\setminus\{0_{\R^3}\}$ and let $\xi^n\in\Q^3$, for $n\geq 2$.
Then, there exist $n$ points $x^1,\dots,x^n$ in $\R^3$ such that
\beq\label{xiixi}
\ba{rll}
\forall\,i\neq j\in\{1,\dots,n\}, & \Pi(x^i+\R\,\xi^i)\cap \Pi(x^j+\R\,\xi^j)={\rm\O}, & \mbox{if }\xi^n\neq 0_{\R^3},
\\ \ecart
\forall\,i\neq j\in\{1,\dots,n-1\}, &
\left\{\ba{l}
\Pi(x^i+\R\,\xi^i)\cap \Pi(x^j+\R\,\xi^j)={\rm\O}
\\ \ecart
\Pi(x^i+\R\,\xi^i)\cap\{0_{Y_3}\}={\rm\O},
\ea\right.
& \mbox{if }\xi^n=0_{\R^3},
\ea
\eeq
where $\Pi$ denotes the canonical surjection from $\R^3$ on $Y_3$.
\end{alem}
\par\noindent
\begin{proof}{ of Theorem~\ref{thm.Cb3d}}
Let $\P(\xi^1,\dots,\xi^n)$ be a closed convex polyhedron of $\R^3$ with rational vertices, and let $\xi^1,\dots,\xi^{n-1}$ be non null vectors of $\R^3$.
\par\smallskip\noindent
{\it First case: $\xi^n\neq 0_{\R^3}$.}
\par\smallskip\noindent
Consider $n$ points $x^1,\dots,x^n\in\R^3$ satisfying the first assertion of \eqref{xiixi}.
For each $i\in\{1,\dots,n\}$, let $C_i$ be the closed cylinder of $\R^3$ of axis $x^i+\R\,\xi^i$ and of radius $R>0$. In view of \eqref{xiixi} we may choose $R$ small enough such that
\beq\label{CiCj}
\forall\,i\neq j\in\{1,\dots,n\},\ \forall\,k\neq\ell\in\Z^3,\quad (k+C_i)\cap (\ell+C_i)={\rm\O}\quad\mbox{and}\quad \Pi(C_i)\cap\Pi(C_j)={\rm\O}.
\eeq
\par\noindent
Figure~\ref{fig1} represents the $n=4$ two-by-two disjoint ``cylinders" of $Y_3$
\[
A_1=\Pi(C_1),\quad B_1=\Pi(C_2),\quad A_2=\Pi(C_3),\quad A_3=\Pi(C_4),
\]
such that the cylinders $C_1,C_2,C_3,C_4$ of $\R^3$ have respective directions $e_1,e_1,e_2,e_3$.
\par\noindent
Moreover, note that each ``cylinder" $\Pi(C_i)$ for $i\in\{1,\dots,n\}$, is a compact set of $Y_3$, since
\[
\Pi^{-1}(\Pi(C_i))=\bigcup_{k\in\Z^3}(k+C_i)
\]
is a closed set of $\R^3$ due to the first equality of \eqref{CiCj}, and thus $\Pi(C_i)$ is a closed set of $Y_3$.
This combined with the second equality of \eqref{CiCj} implies that the two-by-two disjoint compact sets $\Pi(C_1),\dots,\Pi(C_n)$ are to a positive distance from each other.
Hence, we may consider a partition of unity $(\ph_1,\dots,\ph_n)$ in $C^1_\sharp(Y_3)$ associated with $\Pi(C_1),\dots,\Pi(C_n)$, satisfying
\beq\label{phii}
\left\{\ba{cll}
0\leq\ph_i\leq 1 & \mbox{in }Y_3, & \mbox{for any }i\in\{1,\dots,n\},
\\ \ecart
\ph_i=1 & \mbox{in }\Pi(C_i), & \mbox{for any }i\in\{1,\dots,n\},
\\ \ecart
\dis \sum_{i=1}^n\ph_i=1 & \mbox{in }Y_3. &
\ea\right.
\eeq
Then, define the vector field $b\in C^1_\sharp(Y_3)^3$ by
\beq\label{bphipsi1}
b(y):=\sum_{i=1}^n\ph_i(y)\,\,\xi^i\quad\mbox{for }y\in Y_3,
\eeq
and define the probability measures $\mu_1,\dots,\mu_n$ on $Y_3$ by
\beq\label{mui}
d\mu_i(y):={\cha_{\Pi(C_i)}(y)\over |\Pi(C_i)|}\,dy\quad\mbox{for }i\in\{1,\dots,n\}.
\eeq
Let $i\in\{1,\dots,n\}$, and let $\psi\in C^\infty_\sharp(Y_3)$.
By \cite[Lemma~3.5]{Bri1} the function $\psi$ can be represented as a periodized function $\varphi_\sharp$, according to \eqref{tmu}, of a suitable function $\ph\in C^\infty_c(\R^d)$.
Then, using successively \eqref{bphipsi1}, the first equality of \eqref{CiCj} which implies that
\[
\cha_{\Pi(C_i)}=[\cha_{C_i}]_\sharp=\sum_{k\in\Z^3}\cha_{k+C_i}\quad\mbox{in }\R^3,
\]
and integrating by parts on each cylinder $k+C_i$, we get that
\[
\ba{l}
\dis \int_{Y_3}b(y)\cdot\nabla\psi(y)\,d\mu_i(y)={1\over |\Pi(C_i)|}\int_{Y_3}[\cha_{C_i}]_\sharp(y)\,\xi^i\cdot\nabla\varphi_\sharp(y)\,dy
\\ \ecart
\dis ={1\over |\Pi(C_i)|}\sum_{k\in\Z^3}\left(\int_{Y_3}[\cha_{C_i}]_\sharp(y)\,\xi^i\cdot\nabla\varphi(y+k)\,dy\right)
={1\over |\Pi(C_i)|}\int_{\R^3}[\cha_{C_i}]_\sharp(x)\,\xi^i\cdot\nabla\varphi(x)\,dx
\\ \ecart
\dis ={1\over |\Pi(C_i)|}\int_{\R^3}\Big(\sum_{k\in\Z^3}\cha_{k+C_i}(x)\Big)\,\xi^i\cdot\nabla\varphi(x)\,dx
={1\over |\Pi(C_i)|}\sum_{k\in\Z^3}\left(\int_{k+C_i}\xi^i\cdot\nabla\varphi(x)\,dx\right)
\\ \ecart
\dis ={1\over |\Pi(C_i)|}\sum_{k\in\Z^3}\left(\int_{k+\partial C_i}\underbrace{\xi^i\cdot\nu(x)}_{=0}\,\varphi(x)\,d\sigma(x)\right)=0,
\ea
\]
where $\nu(x)$ is the unit outer normal to the cylinder $\partial C_i$ orthogonal to the direction $\xi^i$ of $C_i$.
Hence, by virtue of Proposition~\ref{pro.divcurl} the probability measure $\mu_i$ is invariant for the flow associated with~$b$.
Therefore, by \eqref{bphipsi1} we obtain that
\[
\forall\,i\in\{1,\dots,n\},\quad\mu_i(b)={1\over |\Pi(C_i)|}\int_{\Pi(C_i)}\xi^i\,dx=\xi^i.
\]
By convexity this implies that the rotation set $\sfC_b$ contains the closed polyhedron $\P(\xi^1,\dots,\xi^n)$.
\par
Conversely, by the definition \eqref{bphipsi1} of the vector field $b$ combined with \eqref{phii}, we have for any invariant probability measure $\mu$ for the flow associated with $b$,
\[
\mu(b)=\sum_{i=1}^n\left(\int_{Y_3}\ph_i(y)\,d\mu(y)\right)\xi^i\quad\mbox{with}\quad \sum_{i=1}^n\;\underbrace{\int_{Y_3}\ph_i(y)\,d\mu(y)}_{\in[0,1]}=1.
\]
Hence, the vector $\mu(b)$ belongs to $\P(\xi^1,\dots,\xi^n)$ for any $\mu\in\cI_b$, so that the rotation set $\sfC_b$ agrees with the polyhedron $\P(\xi^1,\dots,\xi^n)$.
This concludes the first case.
\par\medskip\noindent
{\it Second case: $\xi^n=0_{\R^3}$.}
\par\smallskip\noindent
By the second assertion of \eqref{xiixi}, $\Pi(C_1),\dots,\Pi(C_{n-1})$ and $\{0_{Y_3}\}$ are two-by-two disjoint compact sets of $Y_3$, which are thus at a positive distance from each other.
Hence, we may consider a partition of unity $(\ph_1,\dots,\ph_n)$ in $C^1_\sharp(Y_3)$ associated with $\Pi(C_1),\dots,\Pi(C_{n-1})$ and $\{0_{Y_3}\}$ satisfying
\[
\left\{\ba{cll}
0\leq\ph_i\leq 1 & \mbox{in }Y_3, & \mbox{for any }i\in\{1,\dots,n\},
\\ \ecart
\ph_i=1 & \mbox{in }\Pi(C_i), & \mbox{for any }i\in\{1,\dots,n\!-\!1\},
\\ \ecart
\ph_n(0_{Y_3})=1, & &
\\ \ecart
\dis \sum_{i=1}^n\ph_i=1 & \mbox{in }Y_3. &
\ea\right.
\]
Then, define the vector field $b\in C^1_\sharp(Y_3)^3$ by
\[
b(y):=\sum_{i=1}^{n-1}\ph_i(y)\,\,\xi^i\quad\mbox{for }y\in Y_3,
\]
and define the probability measures $\mu_1,\dots,\mu_n$ on $Y_3$ by
\[
d\mu_i(y):={\cha_{\Pi(C_i)}(y)\over |\Pi(C_i)|}\,dy\quad\mbox{for }i\in\{1,\dots,n-1\}\quad\mbox{and}\quad\mu_n:=\delta_{0_{Y_3}}.
\]
Therefore, proceeding as in the first case we obtain that the rotation set $\sfC_b$ agrees with the convex polyhedron $\P(\xi^1,\dots,\xi^{n-1},0_{\R^3})$.
\par
The proof of Theorem~\ref{thm.Cb3d} is now complete.
\end{proof}
\begin{arem}
Contrary to dimension three, the geometrical structure~\eqref{bphipsi1} of the vector field $b$ is necessarily restricted in dimension two to a stratification in some direction $k\in\Z^2\setminus\{0_{\R^2}\}$ by
\[
b(y):=\sum_{i=1}^n\ph_i(y)\,\alpha_i\,k\quad\mbox{for }y\in Y_2,
\]
where $(\ph_1,\dots,\ph_n)$ is a partition of the unity in $C^1_\sharp(Y_2;[0,1])$ associated with $n$ subsets $A_1,\dots,A_n$ of~$Y_2$ which are two-by-two disjoint closure open strips in the direction orthogonal to $k$, and where $\alpha_1,\dots\alpha_n$ are $n$ real constants. See Figure~\ref{fig2} with two strips $A_1,A_2$ in the direction orthogonal to the vector $k:=e_1+e_2$.
This leads us to a closed line segment $\sfC_b$ carried by $[0_{\R^2},k]$.
\end{arem}
\begin{figure}[!t]
\centering
\includegraphics[scale=.4]{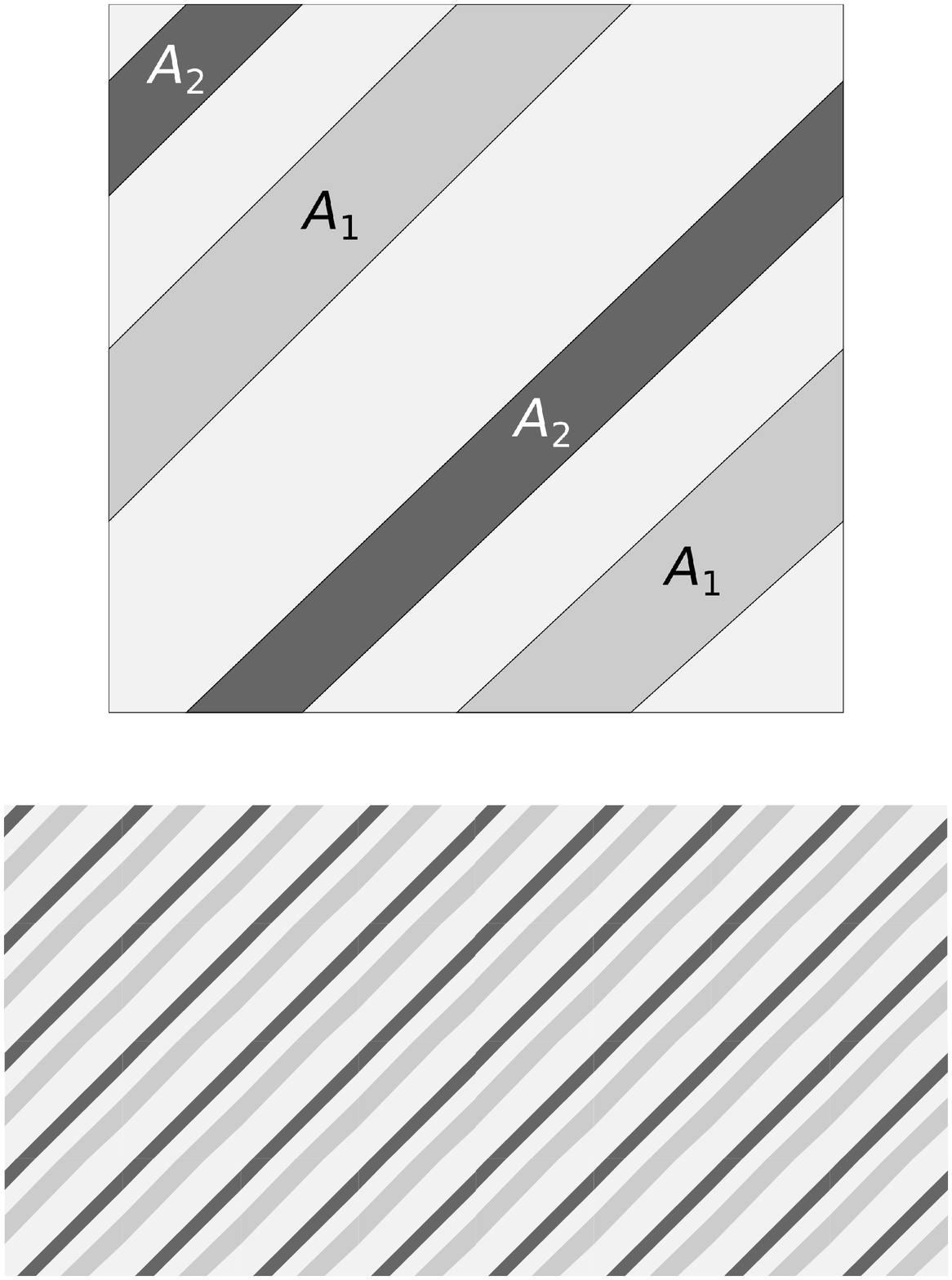}
\newsavebox{\legende}
\sbox{\legende}{\parbox[t]{12.cm}{\it Above two strips being at a positive distance from each other in $Y_2$.
\\
\it Below the periodic repetition in $\R^2$.}}
\caption{\usebox{\legende}}
\label{fig2}
\end{figure}
\begin{proof}{ of Lemma~\ref{lem.xiixi}}
\par\smallskip\noindent
{\it First case: $\xi^n\neq 0_{\R^3}$.}
\par\noindent
Let us prove by induction on $n\geq 2$ that the first assertion \eqref{lem.xiixi}, including in the induction hypothesis the condition
\beq\label{E0}
\forall\,i\in\{1,\dots,n\},\quad x^i\notin E_0:=\bigcup_{k\in\Z^3\setminus\{0_{\R^3}\}}\bigcup_{p\in\Z}\;\big\{x\in\R^3:x\cdot k=p\big\}.
\eeq
The proof of the transition from $n\!-\!1$ to $n$ includes the initialization with $n=2$.
Assume that that the first assertion of \eqref{lem.xiixi} hold for the $n\!-\!1$ vectors $\xi^1,\dots,\xi^{n-1}$ in $\Q^3\setminus\{0_{\R^3}\}$, with $n\!-\!1$ points $x^1,\dots,x^{n-1}$ in $\R^3$ satisfying \eqref{E0}.
\par
First of all, since for any $i\in\{1,\dots,n\!-\!1\}$, for any $k\in\Z^3\setminus\{0_{\R^3}\}$ and for any $p\in\Z$, the hyperplanes of $\R^3$
\[
\big\{x\in\R^3:x\cdot k=p\big\}\quad\mbox{and}\quad\big\{x\in\R^3:(x-x^i)\cdot k=p\big\}
\]
have zero Lebesgue's measure, so have the countable unions $E_0$ in \eqref{E0} and
\beq\label{setE}
E:=\bigcup_{i=1}^{n-1}\bigcup_{k\in\Z^3\setminus\{0_{\R^3}\}}\bigcup_{p\in\Z}\;\big\{x\in\R^3:(x-x^i)\cdot k=p\big\},
\eeq
which implies that $E_0\cup E\neq\R^3$. Then, consider a point $x^n\in\R^3\setminus(E_0\cup E)$, and assume by contradiction that for some $i\in\{1,\dots,n\!-\!1\}$, there exists $x\in\R^3$ such that
\[
\Pi(x)\in \Pi(x^i+\R\,\xi^i)\cap \Pi(x^n+\R\,\xi^n).
\]
Hence, there exist $t_i,t_n\in\R$ and $k^i,k^n\in\Z^3$ such that
\beq\label{xxixn}
x=x^i+t_i\,\xi^i+k^i=x^n+t_n\,\xi^n+k^n.
\eeq
Now, solve the linear system of unknown $(t_i,t_n)$:
\beq\label{titn}
\left\{\ba{l}
\xi^i_1\,t_i-\xi^n_1\,t_n=x^n_1-x^i_1+k^n_1-k^i_1
\\ \ecart
\xi^i_2\,t_i-\xi^n_2\,t_n=x^n_2-x^i_2+k^n_2-k^i_2
\\ \ecart
\xi^i_3\,t_i-\xi^n_3\,t_n=x^n_3-x^i_3+k^n_3-k^i_3,
\ea\right.
\eeq
with $\xi^i,\xi^n, k^i,k^n\in\Q^3$ and $\xi^i,\xi^n$ non null. If the vectors $\xi^i$ and $\xi^n$ are not parallel, then for example,
the $(2\!\times\! 2)$ determinant $\xi^i_2\,\xi^n_3-\xi^i_3\,\xi^n_2$ is non zero, so that the two last equations of \eqref{titn} have a unique solution $(t_i,t_n)$ in $\R^2$.
Putting this solution in the first equation of \eqref{titn} we get that
\beq\label{alphai}
\exists\,(\alpha_1,\alpha_2,\alpha_3)\in\Q^3\setminus\{0_{\R^3}\},\quad \alpha_1\,(x^n_1-x^i_1)+\alpha_2\,(x^n_2-x^i_2)+\alpha_3\,(x^n_3-x^i_3)\in\Q,
\eeq
where $\alpha_1=1$ in the case $\xi^i_2\,\xi^n_3-\xi^i_3\,\xi^n_2\neq 0$.
The other cases are similar. Otherwise, if $\xi^i=\alpha\,\xi^n$ for $\alpha\in\R\setminus\{0\}$, then \eqref{xxixn} yields
\[
x^ i-x^n=(t_n-\alpha\,t_i)\,\xi^n+k^n-k^i,
\]
which still leads us to \eqref{alphai}.
Hence, in all cases we deduce from \eqref{alphai} the existence of a vector $k\in\Z^3\setminus\{0_{\R^3}\}$ such that $(x^n-x^i)\cdot k\in\Z$, which contradicts the choice $x^n\notin E$ defined by \eqref{setE}.
Therefore, we obtain that
\[
\forall\,i\neq j\in\{1,\dots,n-1\},\quad \Pi(x^i+\R\,\xi^i)\cap \Pi(x^n+\R\,\xi^n)={\rm\O},
\]
which combined with the induction hypothesis shows the first assertion of \eqref{xiixi}.
\par\medskip\noindent
{\it Second case: $\xi^n=0_{\R^3}$.}
\par\noindent
Let $i\in\{1,\dots,n\!-\!1\}$. If $0_{Y_3}\in \Pi(x^i+\R\,\xi^i)$, then there exist $t_i\in\R$ and $k^i\in\Z^3$ such that $x^i=t_i\,\xi^i+k^i$.
Hence, due to $\xi^i\neq 0_{\R^3}$, there exists a vector $k\in\Z^3\setminus\{0_{\R^3}\}$ such that $x^i\cdot k\in\Z$, which contradicts the induction hypothesis \eqref{E0} for $n\!-\!1$.
Therefore, the second assertion of \eqref{xiixi} holds, which concludes the proof of Lemma~\ref{lem.xiixi}.
\end{proof}
%%%%%%%%%%
\section{Homogenization of 2D linear transport equations}\label{s.hom}
Let $T\in(0,\infty)$, let $b\in C^1_\sharp(Y_2)^2$, and let $u_0(x,y)\in C^0_c(\R^2;C^0_\sharp(Y_2))$.
Consider the transport equation in the cylinder $Q_T:=(0,T)\times\R^2$
\beq\label{treq}
\left\{\ba{ll}
\dis {\partial u_\ep\over\partial t}(t,x)-b(x/\ep)\cdot\nabla_x u_\ep(t,x)=0 & \mbox{in }Q_T
\\ \ecart
u_\ep(0,x)=u_0(x,x/\ep) & \mbox{for }x\in\R^2.
\ea\right.
\eeq
In the sequel, we define for $f\in L^1_\sharp(Y_2)$, the rescaled function $f_\ep$ by $f_\ep(x):=f(x/\ep)$, $x\in\R^2$. 
\par\medskip
We have the following homogenization result.
\begin{atheo}\label{thm.homtreq}
Let $b$ be a vector field in $C^1_\sharp(Y_2)^2$, and let $\sigma$ be an almost-everywhere positive function in $Y_2$ satisfying
\beq\label{sipr}
\sigma\in W^{1,{2r\over r+2}}_\sharp(Y_2)\;\;\mbox{with}\;\; r\in(2,\infty)\quad\mbox{and}\quad
{1\over\sigma}\in L^p_\sharp(Y_2)\;\;\mbox{with}\;\;p>{r\over r-2},
\eeq
\beq\label{sib}
{\rm div}\,(\sigma\,b)=0\;\;\mbox{in }\cD'(\R^2) \quad\mbox{and}\quad \zeta:=\overline{\sigma\,b}=\int_{Y_2}\sigma(y)\,b(y)\,dy\neq 0_{\R^2}.
\eeq
Then, for any function $u_0(x,y)$ in $C^0_c(\R^2;C^0_\sharp(Y_2))$, the solution $u_\ep$ to equation \eqref{treq} satisfies, up to a subsequence,
\beq\label{U}
\sigma_\ep(x)\,u_\ep(t,x)\rightharpoonup v(t,x):=\overline{U(t,x,\cdot)}\;\;\mbox{weakly in }L^{q}(Q_T),\quad\mbox{with}\;\; q:={p\,r\over p\,r-p-r}\in(1,r),
\eeq
where $U(t,x,y)\in L^q(Q_T;L^q_\sharp(Y_2))$ is solution to the infinite dimensional system of transport equations in $Y_2$
\beq\label{treqU}
\left\{\ba{ll}
\dis {\partial U\over\partial t}(t,x,y)-b(y)\cdot\nabla_y U(t,x,y)=0 & \mbox{in }Y_2
\\ \ecart
{\rm div}_y\big(U(t,x,y)\,b(y)\big)=0& \mbox{in }Y_2,
\ea\right.
\quad\mbox{for a.e. }(t,x)\in Q_T,
\eeq
and the weak limit $v$ of $\sigma_\ep u_\ep$ is solution to the nonlocal homogenized equation
\beq\label{homeqtrvU}
\left\{\ba{ll}
\dis {\partial v\over\partial t}(t,x)-{\zeta\over |\zeta|^2}\cdot\nabla\left(\int_{Y_2}U(t,x,y)\,b(y)\cdot\zeta\,dy\right)=0 & \mbox{in }Q_T
\\ \ecart
v(0,x)=\overline{\sigma\,u_0(x,\cdot)} & \mbox{for }x\in\R^2.
\ea\right.
\eeq
Moreover, if $\sfC_b=\{\zeta\}$, then homogenized equation \eqref{homeqtrvU} becomes the local transport equation
\beq\label{homeqtrv}
\left\{\ba{ll}
\dis {\partial v\over\partial t}(t,x)-\zeta\cdot\nabla v(t,x)=0 & \mbox{in }Q_T
\\ \ecart
v(0,x)=\overline{\sigma\,u_0(x,\cdot)} & \mbox{for }x\in\R^2.
\ea\right.
\eeq
\end{atheo}
\begin{arem}\label{rem.homtreq}
Theorem~\ref{thm.homtreq} is an extension of \cite[Theorem~3.2]{HoXi} and \cite[Theorem~3.2]{BrHe2} obtained with $\sigma=1$, as well as an extension of \cite[Proposition~3.1]{Bri2} obtained with $\sigma$ being the jacobian determinant of some $C^2$-diffeomorphism on $Y_2$. The homogenized transport equation \eqref{homeqtrv} was also derived \cite[Theorem~5.1]{BrHe1} in any dimension under the sole assumption that the rotation set $\sfC_b$ is a unit set, but without oscillating initial condition, {\em i.e.} for $u_0(x,y)$ independent of the fast variable $y$.
\par
In the two-dimensional framework of Theorem~\ref{thm.homtreq}, the case $\#\sfC_b=1$ is regarded as a by-product of the more general limit equation~\eqref{homeqtrvU}. This limit equation illuminates the specific two-dimensional dynamics of Theorem~\ref{thm.FM}, namely the fact that $\sfC_b$ is a closed line segment of $\R^2$.
As a consequence, the drift velocity term in equation~\eqref{homeqtrvU} appears as a nonlocal term as shown by Tartar \cite[Section~2]{Tar} in dimension one (see also \cite{AHZ1,AHZ2} and the references therein).
However, and that is the new point, the nonlocal drift acts in a fixed direction $\zeta\in\sfC_b$.
In other words, the loss of compactness by homogenization of the transport equation \eqref{treq} holds only in the direction of the rotation set $\sfC_b$.
\end{arem}
\begin{proof}{ of Theorem~\ref{thm.homtreq}}
First of all, note that by virtue of \cite[Proposition~II.1, Theorem~II.2]{DiLi} equation~\eqref{treq} has a unique solution $u_\ep$ in $L^\infty(Q_T)$ for a fixed $\ep>0$.
Moreover, using a density argument with $\sigma\in W^{1,2r/(r+2)}_\sharp(Y_2)$ and $u_\ep\in L^\infty(Q_T)$, and the change of test functions $\phi(t,x) \to \phi(t,x)\,\sigma_\ep(t,x)$, the variational formulation of \eqref{treq}
\[
\forall\,\phi\in C^\infty_c([0,T)\times\R^2),\quad
\left\{\ba{l}
\dis -\int_{Q_T}{\partial\phi\over\partial t}(t,x)\,u_\ep(t,x)\,dtdx-\int_{\R^2}\phi(0,x)\,u_0(x,x/\ep)\,dx
\\ \ecart
\dis +\int_{Q_T}{\rm div}\,(\phi(t,x)\,b_\ep(x))\,u_\ep(t,x)\,dtdx=0,
\ea\right.
\]
may be extended to
\[
\forall\,\phi\in C^\infty_c([0,T)\times\R^2),\quad
\left\{\ba{l}
\dis -\int_{Q_T}{\partial\phi\over\partial t}(t,x)\,\sigma_\ep(x)\,u_\ep(t,x)\,dtdx-\int_{\R^2}\phi(0,x)\,\sigma_\ep(x)\,u_0(x,x/\ep)\,dx
\\ \ecart
\dis +\int_{Q_T}{\rm div}\big(\phi(t,x)\,\sigma_\ep(x)\,b_\ep(x)\big)\,u_\ep(t,x)\,dtdx=0.
\ea\right.
\]
Hence, since $\sigma_\ep\,b_\ep$ is divergence free, it follows that for any $\phi\in C^\infty_c([0,T)\times\R^2)$,
\beq\label{ftreq}
\left\{\ba{l}
\dis -\int_{Q_T}{\partial\phi\over\partial t}(t,x)\,\sigma_\ep(x)\,u_\ep(t,x)\,dtdx-\int_{\R^2}\phi(0,x)\,\sigma_\ep(x)\,u_0(x,x/\ep)\,dx
\\ \ecart
\dis +\int_{Q_T}\sigma_\ep(x)\,u_\ep(t,x)\,b_\ep(x)\cdot\nabla_x\phi(t,x)\,dtdx=0,
\ea\right.
\eeq
which implies that the function $v_\ep(t,x):=\sigma_\ep(x)\,u_\ep(t,x)$ is solution to the equation
\beq\label{treqve}
\left\{\ba{ll}
\dis {\partial v_\ep\over\partial t}(t,x)-\sigma_\ep(x)\,b_\ep(x)\cdot\nabla_x u_\ep(t,x)=0 & \mbox{in }Q_T
\\ \ecart
v_\ep(0,x)=\sigma_\ep(x)\,u_0(x,x/\ep) & \mbox{for }x\in\R^2.
\ea\right.
\eeq
Then, multiplying formally $|u_\ep|^{s-2}\,u_\ep$ for $s\geq 1$ in equation \eqref{treqve}, and integrating by parts with $\sigma_\ep\,b_\ep$ divergence free, we get the equality
\[
{d\over dt}\left(\int_{\R^2}\sigma_\ep(x)\,|u_\ep|^s(t,x)\,dx\right)=\int_{\R^2}(\sigma_\ep\,b_\ep)(x)\cdot\nabla\big(|u_\ep|^s(t,x)\big)\,dx=0,
\]
which can be justified by the regularization procedure used in the proof of \cite[Proposition~II.1]{DiLi}.
This combined with $\sigma\in L^1_\sharp(Y_2)$ and $u_0(x,y)\in C^0_c(\R^2;C^0_\sharp(Y_2))$, implies that there exists a constant $c>0$ such that
\beq\label{esieues}
\int_{\R^2}\sigma_\ep(x)\,|u_\ep|^s(t,x)\,dx=\int_{\R^2}\sigma_\ep(x)\,|u_0(x,x/\ep)|^s\,dx\leq c\quad\mbox{for a.e. }t\in(0,T).
\eeq
Next, recalling that $X$ is the flow \eqref{bX} associated with the vector field $b$, then $\ep X(t/\ep,x/\ep)$ is the rescaled flow associated with the rescaled vector field $b_\ep(x)$.
Moreover, using the characteristics method the solution $u_\ep$ to the transport equation \eqref{treq} reads as
\[
u_\ep(x)=u_0\big(\ep X(t/\ep,x/\ep),X(t/\ep,x/\ep)\big)\quad\mbox{for }(t,x)\in Q_T,
\]
with
\[
\forall\,(t,x)\in Q_T,\quad \big|\,\ep X(t/\ep,x/\ep)-x\,\big|=\left|\,\ep\int_0^{t/\ep}b(X(s,x/\ep))\,ds\,\right|\leq \|b\|_{L^\infty(Y_2)^2}.
\]
Hence, since the initial condition $u_0(x,x/\ep)$ is compactly supported independently of $\ep$, so is the function $u_\ep$ in a compact $[0,T]\times K$ of $Q_T$.
Therefore, by the H\"older inequality with the conjugate exponents ${r-1\over q-1}$ and ${r-1\over r-q}$, and by the Sobolev embedding $W^{1,2r/(r+2)}(Y_2)\hookrightarrow L^r_\sharp(Y_2)$ satisfied by $\sigma$ in \eqref{sipr}, from the estimate \eqref{esieues} with $s={q(r-1)\over r-q}$, we deduce  that for a.e. $t\in(0,T)$,
\[
\ba{ll}
\dis \int_{\R^2}\big|\sigma_\ep(x)\,u_\ep(t,x)\big|^{q}\,dx & \dis =\int_{\R^2}(\sigma_\ep(x))^{r(q-1)\over r-1}\,(\sigma_\ep(x))^{r-q\over r-1}\,|u_\ep(t,x)|^q\,dx
\\ \ecart
& \dis \leq \Big(\int_{K}\sigma_\ep^r(x)\,dx\Big)^{q-1\over r-1}\Big(\int_{\R^2}\sigma_\ep(x)\,|u_\ep(t,x)|^{q(r-1)\over r-q}\,dx\Big)^{r-q\over r-1}\leq c,
\ea
\]
which implies that the sequence $v_\ep(t,x):=\sigma_\ep(x)\,u_\ep(t,x)$ is bounded in $L^\infty((0,T);L^{q}(\R^2))$.
\par
On the other hand, by the two-scale convergence of Nguetseng-Allaire \cite{All,Ngu} extended to the $L^r$-spaces in \cite[Section~3]{LNW}, there exists $U(t,x,y)\in L^{q}(Q_T;L^{q}_\sharp(Y_2))$ such that \eqref{U} holds as well as the two-scale limit
\beq\label{2svep}
\lim_{\ep\to 0}\Big(\int_{Q_T}\,v_\ep(t,x)\,\Phi(t,x,x/\ep)\,dt\,dx\Big)=\int_{Q_T\times Y_2}\,U(t,x,y)\,\Phi(t,x,y)\,dtdxdy,
\eeq
for any test function $\Phi\in C^\infty([0,T];C^\infty_c(\R^2;C^\infty_\sharp(Y_2)))$.
\par\noindent
Now, let us follow the two-scale procedure of \cite[Section~2]{HoXi} (see also \cite[Section~2]{BrHe2}).
Testing the variational formulation \eqref{ftreq}, respectively with any function $\phi(t,x)\in C^1_c([0,T)\times\R^2)$, and with function $\phi(t,x)=\ep\,\varphi(t,x)\,\psi(x/\ep)$ for any $\varphi(t,x)\in C^1_c([0,T)\times\R^2)$ and for any $\psi\in C^1_\sharp(Y_2)$, and passing to the two-scale limit as type \eqref{2svep}, we get that, respectively
\beq\label{2slim}
\ba{l}
\dis -\int_{Q_T\times Y_2}{\partial\phi\over\partial t}(t,x)\,U(t,x,y)\,dtdxdy-\int_{\R^2\times Y_2}\phi(0,x)\,u_0(x,y)\,dxdy
\\ \ecart
\dis +\int_{Q_T\times Y_2} b(y)\cdot\nabla_x\phi(t,x)\,U(t,x,y)\,dtdxdy=0,
\ea
\eeq
and
\beq\label{vfdivUb=0}
\ba{l}
\dis \int_{Q_T\times Y_2} \varphi(t,x)\,b(y)\cdot\nabla_y\psi(y)\,U(t,x,y)\,dtdxdy
\\ \ecart
\dis =\int_{Q_T}\varphi(t,x)\Big(\int_{Y_2}U(t,x,y)\,b(y)\cdot\nabla_y\psi(y)\,dy\Big)\,dtdx=0.
\ea
\eeq
Due to the arbitrariness of functions $\ph(t,x)$ and $\psi(y)$, equation \eqref{vfdivUb=0} leads us to (see, {\em e.g.}, \cite[Proposition~1.1]{BrHe2})
\beq\label{divUb=0}
{\rm div}_y(U(t,x,\cdot)\,b)=0\;\;\mbox{in }\cD'(\R^2),\quad\mbox{for a.e. }(t,x)\in(0,T)\times\R^d,
\eeq
or, equivalently, $U(t,x,y)\,dy$ is a signed invariant measure for the flow~$X$ associated with the vector field $b$.
Note that since $1/\sigma\in L^p_\sharp(Y_2)$ by hypothesis, the function $U(t,x,\cdot)/\sigma$ is in $L^{r'}_\sharp(Y_2)$, where by \eqref{U}
\[
{1\over p}+{1\over q}={1\over r'}\,,
\]
while by \eqref{sipr} the function $\sigma$ is in $L^r_\sharp(Y_2)$.
Then, by virtue of \cite[Lemma~2.2]{BrHe2} taking the pair $(\sigma_0,f)=(\sigma,U(t,x,\cdot)/\sigma)$ in the product $L^r_\sharp(Y_2)\times L^{r'}_\sharp(Y_2)$ rather than in $L^1_\sharp(Y_2)\times L^\infty_\sharp(Y_2)$, the function $U(t,x,\cdot)/\si$ is an invariant function for the flow~$X$ (related to Lebesgue's measure), so are the positive and negative parts $U^\pm(t,x,\cdot)/\sigma$ of $U(t,x,\cdot)/\sigma$.
Then, again by \cite[Lemma~2.2]{BrHe2} the measures $U^\pm(t,x,y)\,dy$ are invariant for $X$.
Moreover, applying the collinearity result \eqref{FMcol} with the probability measure $\mu(dx)=\sigma(x)/\overline{\sigma}\,dx$ which by \eqref{sib} is invariant for the flow $X$ and satisfies $\mu(b)=\overline{\sigma\,b}=\zeta\neq 0_{\R^2}$,
we get that $\sfC_b$ is a closed line set carried by a vector line spanned by $\zeta$.
Hence, we deduce that
\beq\label{Upmze}
\overline{U(t,x,\cdot)\,b}=\overline{U^+(t,x,\cdot)\,b}-\overline{U^-(t,x,\cdot)\,b}\;\parallel\zeta\quad\mbox{a.e.}\,(t,x)\in Q_T.
\eeq
Therefore, putting \eqref{Upmze} with the definition \eqref{U} of $v(t,x)$ in the two-scale limit problem \eqref{2slim} we get that
\[
\ba{l}
\dis -\int_{Q_T\times Y_2}{\partial\phi\over\partial t}(t,x)\,v(t,x)\,dtdx-\int_{\R^2\times Y_2}\phi(0,x)\,\overline{\sigma\,u_0(x,\cdot)}\,dx
\\ \ecart
\dis +\int_{Q_T\times Y_2} \nabla_x\phi(t,x)\cdot\Big(\overline{U(t,x,\cdot)\,(b\cdot\zeta)}\;{\zeta\over |\zeta|^2}\Big)\,dtdx=0,
\ea
\]
which is the variational formulation of the desired limit equation \eqref{homeqtrvU}.
\par
Finally, if in addition we assume that $\sfC_b=\{\zeta\}$, condition \eqref{Upmze} becomes the equality
\[
\overline{U(t,x,\cdot)\,b}=\overline{U(t,x,\cdot)}\,\zeta=v(t,x)\,\zeta\quad\mbox{a.e.}\,(t,x)\in Q_T,
\]
which combined with \eqref{homeqtrvU} leads us to the homogenized equation~\eqref{homeqtrv}.
\end{proof}
%%%%%%%%%%

 %%%%%%%%%%
 \section{Appendix}\label{s.appendix}
 \subsection{Proof of \eqref{CbEb}}\label{a.CbEb}
 Let $\xi$ be an extremal point of the convex set $\sfC_b$. By the definition \eqref{Cb} of $\sfC_b$ there exists an invariant probability measure $\mu\in \cI_b$ such that $\mu(b)=\xi$. As a consequence of \eqref{rhob} and \eqref{rhobCb} any limit point of $X(t,x)/t$ as $t\to\infty$, belongs to $\sfC_b$.
Hence, we deduce from Birkhoff's theorem that 
\[
\lim_{t\r\infty} \frac{X(t,x)}{t} = \Delta(x)\quad \mbox{for $\mu$-a.e.~$x\in Y_d$},
\]
where $\Delta$ is some measurable function on $Y_d$ taking its values in $\sfC_b$ and satisfying
\[
\int_{Y_d} \Delta(x)\, d\mu(x) =  \int_{Y_d} b(x)\, d\mu(x) = \xi.
\]
Let us prove that $A := \{x\in Y_d:\Delta(x)=\xi\}$ is a full measure set. To this end, assume by contradiction that $\mu(A) < 1$, and set $A^c := Y_d\setminus A$. From the equality
\[
\xi = \int_{Y_d} \Delta(x)\, d\mu(x) =   \mu(A)\, \xi + \int_{A^c} \Delta(x)\, d\mu(x),
\]
it follows that 
\beq\label{xicA}
\xi = \frac{1}{\mu(A^c)} \int_{A^c} \Delta(x)\, d\mu(x).
\eeq
Note that the set $\sfC_b\setminus\{\xi\}$ is convex, since $\xi$ is an extremal point of $\sfC_b$. Then, considering the supporting hyperplane to the convex set $\sfC_b$ at the point $\xi$, there exists a non null vector $\vartheta\in\R^d$ such that
\[
\forall\,x\in A^c,\quad \left(\xi - \Delta(x)\right) \cdot \vartheta > 0.
\]
Hence, from \eqref{xicA} we deduce that 
\[
0 < \frac{1}{\mu(A^c)} \int_{A^c} \left(\xi - \Delta(x)\right) \cdot \vartheta\,d\mu(x) = \xi \cdot \vartheta - \left(\frac{1}{\mu(A^c)} \int_{A^c} \Delta(x)\, d\mu(x)\right) \cdot \vartheta = 0,
\]
which yields a contradiction. We have just proved that 
\begin{equation}\label{set-A} 
\mu(A) = 1 \quad \mbox{and}\quad \forall\,x\in A,\;\;\lim_{t\r\infty} \frac{X(t,x)}{t} = \xi.
\end{equation}
Next, define the set
\[
B := \left\{x\in Y_d : \lim_{t\r\infty} \frac{X(t,x)}{t} = \xi\right\}.
\]
There exists an ergodic invariant probability measure $\nu\in \cE_b$ such that $\nu(B) = 1$.
Otherwise, we have $\nu(B) = 0$ for any $\nu\in \cE_b$, since the set $B$ is invariant for the flow $X$ \eqref{bX} associated with $b$ (due to the semi-group property \eqref{sgX}).
Then, by virtue of the ergodic decomposition theorem (see, {\em e.g.}, \cite[Theorem~14.2]{Cou}) we get that for any $\mu\in \cI_b$, $\mu(B) = 0$. This contradicts \eqref{set-A}, since we have $A\subset B$ and $\mu(A)=1$.
\par
Finally, applying Birkhoff's theorem with this measure $\nu\in \cE_b$, we get that
\[
\lim_{t\to\infty} \frac{X(t,x)}{t} = \int_{Y_d} b(y)\, \nu(dy)\quad \mbox{for $\nu$-a.e.~$x\in Y_d$},
\]
which implies that $\xi=\nu(b)$.
\cqfd
\subsection{Proof of the measure representation \eqref{bmuus}}\label{app.rep}
\par\noindent
We have chosen to work first in the space $\R^2$ rather than directly in the torus $Y_2$, for treating the equations with no regular solutions in the distributional sense.
\par
First of all, by virtue of the divergence-curl Proposition~\ref{pro.divcurl} the Borel measure $\widetilde{\mu}$ on $\R^2$ defined by \eqref{tmu} satisfies that $b\,\widetilde{\mu}$ is divergence free in $\R^2$.
Hence, there exists a stream function with bounded variation $u\in BV_{\rm loc}(\R^2)$ (see, {\em e.g.}, \cite[Theorem~3.1]{GiRa}), unique up to an additive constant, such that
\beq\label{btmuu}
b\,\widetilde{\mu}=R_\perp\nabla u\quad\mbox{in }\cD'(\R_2).
\eeq
\par
\par\noindent
Next, let us prove that there exists $ u^\sharp\in BV_\sharp(Y_2)$ such that the stream function $u$ of \eqref{btmuu} reads as
\beq\label{uxitu}
u(x)=\xi\cdot x+u^\sharp(x)\;\;\mbox{for a.e. }x\in\R^2,\quad\mbox{where}\quad\xi:=-\,R_\perp\mu(b).
\eeq
Integrating by parts we have for any $k\in\Z^2$ and any $\Phi\in C^\infty_c(\R^2)^2$,
\[
\int_{\R^2}{\rm div}\,(R_\perp\Phi)(x)\,u(x+k)\,dx=\int_{\R^2}{\rm div}\,(R_\perp\Phi)(x+k)\,u(x)\,dx
=\int_{\R^2}\Phi(x+k)\cdot R_\perp d\nabla u(x),
\]
which by \eqref{btmuu} and \eqref{tmu} yields
\[
\int_{\R^2}{\rm div}\,(R_\perp\Phi)(x)\,u(x+k)\,dx=\int_{Y_2}[\Phi(\cdot+k)]_\sharp(y)\cdot b(y)\,d\mu(y)
=\int_{Y_2}\Phi_\sharp(y)\cdot b(y)\,d\mu(y),
\]
which is independent of the integer vector $k$.
Hence, we deduce that
\[
\forall\,\Phi\in C^\infty_c(\R^2)^2,\quad \int_{\R^2}{\rm div}\,(R_\perp\Phi)(x)\,(u(x+k)-u(x))\,dx=0,
\]
which implies that there exists a constant $c_k\in\R$ such that
\beq\label{uck}
u(x+k)-u(x)=c_k\quad\mbox{for a.e. }x\in\R^2.
\eeq
Then, define the vector $\xi\in\R^2$ by
\beq\label{xiu}
\xi\cdot e_i:=c_i=u(\cdot+e_i)-u\quad\mbox{for }i\in\{1,2\}.
\eeq
It easily follows from \eqref{uck} and \eqref{xiu} that
\[
\forall\,k\in\Z^2,\quad u(x+k)-u(x)=\xi\cdot k\quad\mbox{for a.e. }x\in\R^2,
\]
which implies that the function $ u^\sharp$ defined by
\[
u^\sharp(x):=u(x)-\xi\cdot x\quad\mbox{for a.e. }x\in \R^2,
\]
is $\Z^2$-periodic, or equivalently, $ u^\sharp\in L^2_\sharp(Y_2)$.
\par
Next, by \cite[Lemma~2.3]{BrHe1} there exists a function $\varphi\in C^\infty_c(\R^2)$ such that $\varphi_\sharp=1$ according to~\eqref{tmu}.
On the one hand, integrating by parts we have
\[
\ba{l}
\dis \int_{\R^2}\varphi(x)\,R_\perp d\nabla u(x)=\left(\int_{\R^2}\varphi(x)\,dx\right)R_\perp\xi+\int_{\R^2}\varphi(x)\,R_\perp d\nabla u^\sharp(x)
\\ \ecart
\dis =\left(\int_{Y_2}\varphi_\sharp(y)\,dy\right)R_\perp\xi-R_\perp\int_{\R^2}\nabla\varphi(x)\, u^\sharp(x)\,dx,
\ea
\]
which due to the $\Z^2$ periodicity of $ u^\sharp$ yields
\[
\int_{\R^2}\varphi(x)\,R_\perp d\nabla u(x)=R_\perp\xi-R_\perp\int_{\R^2}\underbrace{\nabla\varphi_\sharp(y)}_{=0}\, u^\sharp(y)\,dy=R_\perp\xi.
\]
On the other hand, by \eqref{tmu} and \eqref{btmuu} we have
\[
\int_{\R^2}\varphi(x)\,R_\perp d\nabla u(x)=\int_{\R^2}\varphi(x)\,b(x)\,d\widetilde{\mu}(x)=\int_{Y_2}\underbrace{\varphi_\sharp(y)}_{=1}b(y)\,d\mu(y)=\mu(b).
\]
Therefore, we get that $\xi=-\,R_\perp\mu(b)$, which leads us to \eqref{uxitu}.
\par
Finally, using successively equality \eqref{uxitu}, integrations by parts over $Y_2$ with $\Z^2$-periodic functions and over $\R^2$ with compact support functions in $\R^2$, and equality \eqref{btmuu}, we get that for any vector-valued function $\Phi\in C^\infty_c(\R^2)^2$,
\[
\ba{l}
\dis \int_{Y_2}\Phi_\sharp(y)\cdot R_\perp d\nabla u(y)=
\int_{Y_2}\Phi_\sharp(y)\cdot R_\perp\xi\,dy+\int_{Y_2}\Phi_\sharp(y)\cdot R_\perp d\nabla u^\sharp(y)
\\ \ecart
\dis =\int_{Y_2}\Phi_\sharp(y)\cdot R_\perp\xi\,dy+\int_{Y_2}{\rm div}\,(R_\perp\Phi_\sharp)(y)\, u^\sharp(y)\,dy
\\ \ecart
\dis =\int_{\R^2}\Phi(x)\cdot R_\perp\xi\,dy+\int_{\R^2}{\rm div}\,(R_\perp\Phi)(x)\, u^\sharp(x)\,dx
\\ \ecart
\dis =\int_{\R^2}{\rm div}\,(R_\perp\Phi)(x)\,(\xi\cdot x+ u^\sharp(x))\,dx=\int_{\R^2}{\rm div}\,(R_\perp\Phi)(x)\,u(x)\,dx
\\ \ecart
\dis =\int_{\R^2}\Phi(x)\cdot R_\perp d\nabla u(x)=\int_{\R^2}\Phi(x)\cdot b(x)\,d\widetilde{\mu}(x)
=\int_{Y_2}\Phi_\sharp(y)\cdot b(y)\,d\mu(y),
\ea
\]
namely,
\[
\forall\,\Phi\in C^\infty_c(\R^2)^2,\quad \int_{Y_2}\Phi_\sharp(y)\cdot R_\perp d\nabla u(y)=\int_{Y_2}\Phi_\sharp(y)\cdot b(y)\,d\mu(y).
\]
This combined with the representation of any smooth function in $C^\infty_\sharp(Y_2)^2$ by $\Phi_\sharp$ for a suitable function $\Phi\in C^\infty_c(\R^2)^2$ (see \cite[Lemma~2.3]{BrHe1}), thus yields by \eqref{uxitu}
\[
b\,\mu=R_\perp\nabla u=\mu(b)+R_\perp\nabla u^\sharp\quad \mbox{in }Y_2,
\]
which is the desired representation \eqref{bmuus}.
\cqfd
%%%%%%%%%%
\end{document}